\documentclass[a4paper,11pt]{amsart}

\usepackage{epsfig}
\usepackage{amsthm,amsfonts}
\usepackage{amssymb,graphicx,color}
\usepackage{float}
\graphicspath{ {Images/} }
\usepackage[labelfont={bf},textfont={it}]{caption}
\usepackage[labelfont={bf, it}]{subcaption}
\usepackage[all]{xy}
\usepackage{verbatim}
\usepackage{hyperref}
\usepackage{enumitem}
\usepackage{tikz, tkz-fct}

\makeatletter
\newcommand{\mlabel}[2]{\def\@currentlabel{#2}\label{#1}} 
\makeatother

\newtheorem{theorem}{Theorem}[section]
\newtheorem*{theorem*}{Theorem}
\newtheorem{lemma}[theorem]{Lemma}
\newtheorem{corollary}[theorem]{Corollary}
\newtheorem{proposition}[theorem]{Proposition}
\newtheorem{definition}[theorem]{Definition}
\newtheorem{example}[theorem]{Example}
\newtheorem{remark}[theorem]{Remark}

\newtheorem{claim}{Claim}[theorem]
\newtheorem{exercise}{Exercise}[section]

\newcommand{\C}{\mathbb C}
\newcommand{\R}{\mathbb R}

\newcommand{\Z}{\mathbb Z}
\newcommand{\N}{\mathbb N}

\begin{document}
\title{Bi-Lipschitz invariance of the multiplicity}

\author[A. Fernandes]{Alexandre Fernandes}
\author[J. E. Sampaio]{Jos\'e Edson Sampaio}

\address{Alexandre Fernandes: Departamento de Matem\'atica, Universidade Federal do Cear\'a, Av. Humberto Monte, s/n Campus do Pici - Bloco 914, 60455-760, Fortaleza-CE, Brazil. E-mail: {\tt alex@mat.ufc.br}}
\address{J. Edson Sampaio: Departamento de Matem\'atica, Universidade Federal do Cear\'a, Rua Campus do Pici, s/n, Bloco 914, Pici, 60440-900, Fortaleza-CE, Brazil. E-mail: {\tt edsonsampaio@mat.ufc.br}}

\thanks{The first named author was partially supported by CNPq-Brazil grant 304700/2021-5. The second named author was supported by the Serrapilheira Institute (grant number Serra -- R-2110-39576) and was partially supported by CNPq-Brazil grant 310438/2021-7.}

%

%
%
%
%
%
%

\maketitle

\begin{abstract}
The multiplicity of an algebraic curve $C$ in the complex plane at a point $p$ on that curve is defined as the number of points that occur at the intersection of $C$ with a general complex line that passes close to the point $p$. It is shown that $p$ is a singular point of the curve $C$ if and only if this multiplicity is greater than or equal to 2, in this sense, such an integer number can be considered as a measure of how singular can be a point of the curve $C$. In these notes, we address the classical concept of multiplicity of singular points of complex algebraic sets (not necessarily complex curves) and we approach the nature of the multiplicity of singular points as a geometric invariant from the perspective of the Multiplicity Conjecture (Zariski 1971). More precisely, we bring a discussion on the recent results obtained jointly with Lev Birbrair, Javier Fern\'andez de Bobadilla, L\^e Dung Trang and Mikhail Verbitsky on the bi-Lipschitz invariance of the multiplicity.
\end{abstract}

\tableofcontents
\newpage

\section{Introduction}

Unless explicitly mentioned to the contrary, all the analytic subsets of $\C^n$ considered here are closed subsets of $\C^n$.

\subsection{Local Analytic Structure} Let $p\in X\subset\C^n$ and $q\in Y\subset\C^m$ be analytic subsets. We say that the pair $(X,p)$ is {\bf analytic equivalent}\index{equivalent!analytic} to $(Y,q)$ if there exist neighbourhoods $U\subset\C^n$ of $p$ and $V\subset\C^m$ of $q$ and an analytic mapping $F\colon U\cap X\rightarrow V\cap Y$; $F(p)=q$ with inverse map $G\colon V\cap Y\rightarrow U\cap X$ also analytic. This definition gives us an equivalence relation; each equivalence class is what we call a {\bf local analytic structure}\index{structure!local analytic}. We establish that the equivalence class of $(\C^n,0)$ is the {\bf regular local analytic structure}\index{structure!regular local analytic} in dimension $n$. Local Analytic Geometry is the research field in Mathematics in charge of describing all local analytic structures.

Given $p\in X\subset\C^n$ an analytic subset, we denote by $\mathcal{O}_{X,p}$  the set of analytic functions defined in some neighbourhood of $p$ in $X$ equipped with natural binary operations of addition and multiplication. Defined in that way, $\mathcal{O}_{X,p}$ is a local ring with maximal ideal given by
$$\mathcal{M}_{X,p}=\{f\in \mathcal{O}_{X,p} \ : \ f(p)=0\}.$$

Next result makes a bridge connecting Local Analytic Geometry with Commutative Algebra (see \cite{Chirka:1989}).

\begin{theorem}
Let $p\in X\subset\C^n$ and $q\in Y\subset\C^m$ be analytic subsets. The pair $(X,p)$ defines the same local analytic structure\index{structure!local analytic} as $(Y,q)$ if, and only if, $\mathcal{O}_{X,p}$ is isomorphic to $\mathcal{O}_{Y,q}$ as local $\C$-algebras.
\end{theorem}

As already mentioned here, we observe that this theorem above provides a way to study the classification of local analytic structures from the algebraic point of view. Next, from the algebraic point of view, we present one of the more important invariant of local analytic structures.

\begin{proposition}
Let $X\subset\C^n$ be an analytic subset such that in a neighbourhood of $p\in X$ it has pure dimension $d$. There exists a polynomial $P(t) \in \C [t]$ of degree $d$ such that $P(k)$ is  equal to $\dim \mathcal{O}_{X,p}/\mathcal{M}^k_{X,p}$ as a $\C$ vector space, for $k$ large enough. Moreover, the leading coefficient of $P(t)$ times $d!$ is equal to some positive integer $e$.
\end{proposition}

\begin{definition}
The polynomial $P(t)$ is called {\bf the Hilbert-Samuel Polynomial}\index{polynomial!Hilbert-Samuel} of the pair $(X,p)$. The integer number $e$ provided in the above proposition is called {\bf multiplicity of $X$ at $p$}\index{multiplicity} and denoted by $m(X,p)$.
\end{definition}

\begin{example}
Let us show that $m(\C^n,0)=1$. Actually, we know that $\mathcal{O}_{\C^n,0}$ is isomorphic to $\C\{z_1,\dots ,z_n\}$, so, it has pure dimension $n$. Hence, $\mathcal{M}^k_{\C^n,0}$ is the ideal of $\C\{z_1,\dots ,z_n\}$ generated by $z_1^{a_1}\cdots z_n^{a_n}$ where $a_1,\dots ,a_n$ are non-negative integer numbers such that $a_1+\cdots +a_n =k$ and, therefore,
\begin{eqnarray*}
	\dim \mathcal{O}_{\C^n,0}/\mathcal{M}^k_{\C^n,0} &=& \frac{1}{n!}(n+k-1)(n+k-2)\cdots (k+1)(k)
\end{eqnarray*}
$\therefore$ $m(\C^n,0)=1$.
\end{example}

\begin{example}
Let $X\subset\C^2$ be the cusp defined by
$$X=\{(z_1,z_2)\in\C^2 \ : \ z_1^3=z_2^2\}.$$ In this case, we see that $m(X,0)=2$. Indeed, $\mathcal{O}_{X,0}$ is isomorphic to $\C\{z_1,z_2\}/I$ where $I$ is the ideal of $\C\{z_1,z_2\}$ generated by $z_1,z_2$ with the following relation $z_1^3=z_2^2$. In other words, we have that $\mathcal{O}_{X,0}$ is isomorphic to
$$f(z_1)+g(z_1)z_2 \ : \ f(z_1),g(z_1)\in\C\{z_1\}.$$ Hence, $$\dim \mathcal{O}_{X,0}/\mathcal{M}_{X,0}^k = 2k-1,$$ and as $(X, 0)$ has pure dimension
1, we get $m(X,0)=2$.
\end{example}

Next result says that multiplicity of points is an invariant of the local analytic structure\index{structure!local analytic}.

\begin{theorem}
Let $p\in X\subset\C^n$ and $q\in Y\subset\C^m$ be analytic subsets. If $(X,p)$ defines the same local analytic structure as $(Y,q)$, then $m(X,p)=m(Y,q)$.
\end{theorem}

\begin{proof}
Let us assume that $(X,p)$ defines the same local analytic structure\index{structure!local analytical} as $(Y,q)$, hence $\mathcal{O}_{X,p}$ is isomorphic to $\mathcal{O}_{Y,q}$ as local $\C$-algebra. Then, $$\dim \mathcal{O}_{X,p}/\mathcal{M}^k_{X,p}=\dim \mathcal{O}_{Y,q}/\mathcal{M}^k_{Y,q} \ \forall k,$$ and $(X,p)$ and $(Y,q)$ have the same Hilbert-Samuel Polynomial. In particular, $m(X,p)=m(Y,q)$.
\end{proof}

Next we see a more geometric way to get the multiplicity of points (see \cite{JongP:2000} and \cite{Chirka:1989}).

\begin{proposition}
Let $X\subset\C^n$ be an analytic subset such that in a neighbourhood of $p\in X$ it has dimension $d$. If $L\colon\C^n\rightarrow\C^d$ is a generic linear projection,  then its restriction to $X\cap U$  defines a finite mapping of topological degree $m$ for each small enough neighbourhood $U\subset\C^n$ of the point $p$. Moreover, if $X\subset\C^n$ is an analytic subset such that in a neighbourhood of $p\in X$ it has pure dimension $d$ then the integer number $m$ is equal to the multiplicity $m(X,p)$. 
\end{proposition}

In the above proposition, let us make clear that a linear projection $L\colon\C^n\rightarrow\C^d$ is generic when the intersection of $Ker(L)$ with the tangent cone $C(X,p)$ (see Section \ref{sec:tg_cones}) is only the null vector.

\begin{example}
Let $X\subset\C^n$ be a hypersurface defined as the zero set of an analytic function $f\colon U\subset\C^n\rightarrow\C$ in a neighbourhood $U$ of the origin $0\in\C^n$; $f(0)=0$. Let us write $$f(z)=f_m(z)+f_{m+1}(z)+\cdots+f_{k}(z)+\cdots $$ where each $f_k(z)$ is a homogeneous polynomial of degree $k$ and $f_m\not\equiv 0$. In this case, $m(X,0)=m$.
\end{example}

\begin{definition}
 We say that $p\in X \subset \C^n$ is a {\bf regular}\index{regular} point\index{point!regular} (of $X$) if $\mathcal{O}_{X,p}\cong \mathcal{O}_{\C^k,0}$. A point $p\in X$ is {\bf singular}\index{singular}\index{point!singular} if it is not regular.
\end{definition}

\begin{corollary}\label{cor:regularity_vs_mult}
	Let $X\subset\C^n$ be an analytic subset such that in a neighbourhood of $p\in X$ it has pure dimension $d$. Then, $p$ is a regular point of $X$ if and only if $m(X,p)=1$.
\end{corollary}

Although multiplicity is enough to identify the regular local analytic structures, the following example shows us that discrete invariants are not enough to describe all local analytic structures.

\begin{example}
	(Whitney 1965) Let us consider the following family of four lines through the origin $$X_t \colon \  x\cdot y\cdot (y-x)\cdot (y-tx)=0.$$ For generic $s\neq t$, H. Whitney (see \cite{Whitney:1965b}) noted that the local analytic structures $(X_s,0)$ and $(X_t,0)$ are not analytic equivalent.
\end{example}

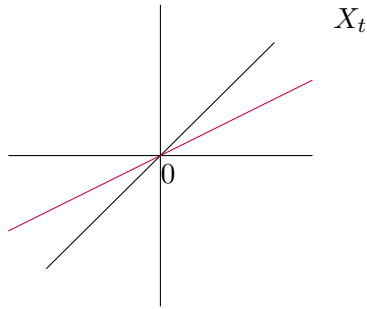
\begin{center}
\begin{tikzpicture}
\draw (-2,0) -- (2,0);
\draw (-1.5,-1.5) -- (1.5,1.5);
\draw (0,-2) -- (0,2);
\draw[purple] (-2,-1) -- (2,1);
\coordinate [label= 0] (0) at (0.1,-0.5);
\coordinate[label= $X_t$] (1) at (2.5,1.5);
\end{tikzpicture}
\end{center}


\subsection{Local Topological Structure}
Let $p\in X\subset\C^n$ and $q\in Y\subset\C^n$ be analytic subsets. We say that the pair $(X,p)$ is {\bf topological equivalent}\index{equivalent!topological} to $(Y,q)$ if there exist neighbourhoods $U\subset\C^n$ of $p$ and $V\subset\C^m$ of $q$ and a homeomorphism $F\colon U\cap X\rightarrow V\cap Y$; $F(U\cap X)=V\cap Y$ and $F(p)=q$. This definition give us an equivalence relation; each equivalence class is what we call a {\bf local topological structure}\index{structure!local topological}. We establish that the equivalence class of $(\C^n,0)$ is the {\bf regular local topological structure}\index{structure!regular local topological} in dimension $n$.

In the 60's, results that pointed to the possibility of describing the local topology of analytical sets only with discrete invariants greatly boosted research on the subject and attracted eminent mathematicians to questions that are still open today. Actually, it was proved that there are only countably many infinite local topological structures and,  in the direction of looking for nice discrete invariants of the local topological structures, in 1971 O. Zariski (see \cite{Zariski:1971}) placed the following question which is still unanswered 

\begin{center}[O. Zariski 1971]
	{\it Let $X,Y\subset\C^n$ be analytic subsets of codimension 1 such that $(X,p)$ and $(Y,q)$ are topological equivalent, is it true that $m(X,p)$ must be equal to $m(Y,q)$?}
\end{center}


\subsection{Local Lipschitz Structure}
Let $p\in X\subset\C^n$ and $q\in Y\subset\C^n$ be analytic subsets. We say that the pair $(X,p)$ is {\bf bi-Lipschitz equivalent}\index{equivalent!bi-Lipschitz} to $(Y,q)$ if there exist neighbourhoods $U\subset\C^n$ of $p$ and $V\subset\C^m$ of $q$ and a Lipschitz map $F\colon U\cap X\rightarrow V\cap Y$; $F(p)=q$ with inverse map $G\colon V\cap Y\rightarrow U\cap X$ also Lipschitz, i.e. there exists a positive constant $\lambda\geq 1$ such that the bi-univocal correspondence $F$ satisfies:

$$ \frac{1}{\lambda} |x_1-x_2|\leq |F(x_1)-F(x_2)|\leq \lambda |x_1-x_2| \ \forall x_1,x_2\in X\cap U.$$

This definition give us an equivalence relation; each equivalence class is what we call a {\bf local Lipschitz structure}\index{structure!local Lipschitz}. We establish that the equivalence class of $(\C^n,0)$ is the {\bf regular local Lipschitz structure}\index{structure!regular local Lipschitz} in dimension $n$. Lipschitz Geometry of Singularities is the research field in Mathematics in charge of describing all local Lipschitz structures.

In the mid 80's (see \cite{Mo}), Tadeusz Mostowski proved that  there are only countably many infinite local Lipschitz structures (Parusinski proved the real version of such result in \cite{Parusinski:1994}); this result greatly stimulated research on the Lipschitz geometry of singularities, mainly in looking for discrete invariants for the respective classification problem. In this direction, the main objective of these notes is to bring a survey of results, some of them published with other collaborators, on the following issues:


\begin{enumerate}[leftmargin=0pt]
	\item[]{\bf Question AL($d$):} Let $X\subset \C^n$ and $Y\subset \C^m$ be two complex analytic sets with $\dim X=\dim Y=d$, $0\in X$ and $0\in Y$. 
	If there exists a bi-Lipschitz 
	homeomorphism $\varphi\colon (X,0)\to (Y,0)$, then is $m(X,0)=m(Y,0)$?
	\item[]{\bf Question  AL($n,d$):} Let $X,Y\subset \C^n$ be two complex analytic sets with $\dim X=\dim Y=d$ and $0\in X\cap Y$. 
	If there exists a bi-Lipschitz homeomorphism $\varphi\colon (\C^n,X,0)\to (\C^n,Y,0)$ (i.e., $\varphi\colon (\C^n,0)\to (\C^n,0)$ is a bi-Lipschitz homeomorphism such that $\varphi(X)=Y$), then is $m(X,0)=m(Y,0)$?
\end{enumerate}


We call each one of the above questions of {\bf bi-Lipschitz invariance of the multiplicity problem}\index{multiplicity!bi-Lipschitz invariance}. We emphasize the importance of citing some of the main references of pioneering results in the proposed investigation of multiplicity as an invariant of local structures that are less rigid than analytic (like Lipschitz structures), namely: in the paper \cite{Gau-Lipman:1983}, Gau and Lipman proved the invariance of multiplicity for differentiable local structures (Ephraim, in \cite{Ephraim:1976a}, addressed the case of continuously differentiable local structures, see also the result proved by Trotman in \cite{Trotman:1977}) and, in the paper \cite{Comte:1998}, assuming severe restrictions on the Lipschitz constants that conjugate two pairs $(X,p)$ and $(Y,q)$, Comte showed that $m(X, p)=m(Y,q)$.


In recent works written in collaboration with Birbrair (UFC), Fern\'andez de Bobadilla (BCAM), L\^e (Aix-Marseille) and Verbitsky (IMPA), we proved the following results.

\begin{description}
	\item[1] Regular local Lipschitz structure is equivalent to regular local analytic structure  (see \cite{BirbrairFLS:2016} and \cite{Sampaio:2016});
	\item[2] In dimension 2, the multiplicity is an invariant of the local Lipschitz structure (see \cite{BobadillaFS});
	\item[3] In dimension greater than 2, the multiplicity is not an invariant of the local Lipschitz structure (see \cite{BirbrairFSV:2020}).
\end{description}

Pham and Teissier in \cite{P-T}, with contributions
of Fernandes in \cite{F} and Neumann and Pichon in \cite{N-P}, proved that multiplicity is an invariant of the local Lipschitz structure in dimension 1 (see also \cite{FSS}). 
Therefore, Results 2 and 3 above, together with the result of Pham and Teissier, can be summarized as follows: {the multiplicity is invariant of the local Lipschitz structure only in dimensions 1 and 2}.


Let us describe how these notes are organized. Section 2 contains preliminary results where we introduce the concept of Lipschitz normally embedded sets and we present the Pancake Decomposition Theorem. Section 3 is devoted to explore the notion of tangent vectors, more precisely, we introduce the concept of tangent cone to singular points and, among other results, we present a proof of the bi-Lipschitz invariance of the tangent cone of subanalytic singularities. In Section 4, we recover the notion of regular local Lipschitz structure and we introduce other notions of local Lipschitz regularities; we also present a proof that Lipschitz regularity of  a  local analytic structure implies that such structure must be analytic regular. In Section 5, we introduced the so-called relative multiplicities, we present a proof that such multiplicities are bi-Lipschitz invariant and we show that Question AL($d$) (resp. AL($n,d$)) has a positive answer if, and only if, it has a positive answer for irreducible homogeneous algebraic singularities. Finally, in Section 6 we address the problem of the bi-Lipschitz invariance of the multiplicity.

\section{Lipschitz normally embedded sets}
In this section we define Lipschitz normally embedded sets and we present some important properties of this notion.

Let $Z\subset\R^n$ be a path connected subset. Given two points $q,\tilde{q}\in Z$, we define the {\bf inner distance}\index{inner distance} on $Z$ between $q$ and $\tilde{q}$ by the number:
$$d_Z(q,\tilde{q}):=\inf\{ \mbox{Length}(\gamma) \ | \ \gamma \ \mbox{is an arc on} \ Z \ \mbox{connecting} \ q \ \mbox{to} \ \tilde{q}\}.$$
\begin{definition}
We say that $Z$ is {\bf Lipschitz normally embedded (LNE)}\index{LNE}\index{Lipschitz normally embedded}, if there is a constant $C\geq 1$ such that $d_Z(q,\tilde{q})\leq C\|q-\tilde{q}\|$, for all $q,\tilde{q}\in Z$. We say that $Z$ is {\bf Lipschitz normally embedded at $p$} (shortly LNE at $p$), if there is a neighbourhood $U$ such that $p\in U$ and $Z\cap U$ is an LNE set or, equivalently, that the germ $(Z,p)$ is LNE. In this case, we also say that $Z$ is $C$-LNE (resp. $C$-LNE at $p$). 
\end{definition}

\begin{proposition}\label{lne:exer:lne_lip_inv}
Let $X\subset \R^n$ and $Y\subset \R^m$ be non-empty subsets. Assume that there exists a bi-Lipschitz homeomorphism $\psi\colon X\to Y$. Then, $X$ is LNE\index{LNE} at $x_0\in X$ if and only if $Y$ is LNE at $\psi(x_0)$.
\end{proposition}
\begin{proof}
The proof is left as an exercise for the reader.
\end{proof}

Let us recall the following ``Pancake decomposition''\index{Pancake decomposition} result:
\begin{lemma}[Proposition 3 in \cite{Kurdyka:1997}]\label{lne:lem:lne_decomp}
Let $X\subset\R^m$ be a subanalytic set and $\varepsilon>0$. Then for each $p$ in the closure of $X$, denoted by $\overline{X}$, there exist $\delta>0$ and a finite decomposition 
$X\cap B_{\delta}(p) = \bigcup_{j =1}^k \Gamma_j$  such that:
\begin{enumerate}
\item each $ \Gamma_j$ is a subanalytic connected analytic submanifold of  
$\R^m$,
\item each $ \overline \Gamma_j$ satisfies $d_{\overline \Gamma_j}(x,y)\leq(1+\varepsilon)\|x-y\|$ for any $x,y\in\overline \Gamma_j$. 
\end{enumerate}
\end{lemma}

The above result has the following consequence:

\begin{proposition}\label{lne:prop:top_inner} Let $X\subset \R^m$ be a subanalytic set. Then $d_{X}$ induces the same topology on $X$ as the topology induced by the standard topology on $\R^m$.
\end{proposition}
\begin{proof}
The proof is left as an exercise for the reader.
\end{proof}

\subsection{Exercises}\label{sec:exercises_preliminaries}

\begin{exercise}
Prove that any connected compact $C^1$ submanifold of $\R^n$ is LNE.
\end{exercise}

\begin{exercise}\label{lne:prop:inner_equals_geodesic_dist_smooth} Let $X\subset \C^n$ be a complex analytic set. Then for any $x,y\in X$, we have that 
$$
d_{X}(x,y)=\inf\limits_{\beta\in \Omega(x,y)}\int_0^1\|\beta'(t)\|dt,
$$
where $\Omega(x,y)=\{\beta\colon[0,1]\to X; \beta(0)=x,\, \beta(1)=y\mbox{ and }\beta \mbox{ is piecewise }C^1\}$. 
\end{exercise}

\begin{exercise}
Prove Proposition \ref{lne:exer:lne_lip_inv}.
\end{exercise}

\begin{exercise}
Prove Proposition \ref{lne:prop:top_inner}.
\end{exercise}

\begin{exercise}
Let $X\subset \C^2$ be a complex analytic curve. Let $X_1,...,X_r$ be the irreducible components of $X$ (at $0$). Then, $X$ is LNE\index{LNE} at $0$ if and only if each $X_i$ is smooth at $0$ and for $i\not=j$, $X_j$ and $X_i$ meet transversally at $0$.
\end{exercise}

\begin{exercise}
Let $X\subset \C^m$ be a complex algebraic set. Assume that $d_X(x,y)=\|x-y\|$ for all $x,y\in X$. Prove that $X$ is an affine linear subspace of $\C^m$.
\end{exercise}

\section{Tangent cones}\label{sec:tg_cones}

We begin this section with the definition of tangent cone of a subset $X\subset \mathbb{K}^{m}$ at a point, where $\mathbb{K}$ is $\R$ or $\C$.

\begin{definition}
Let $X\subset \mathbb{K}^{m}$ be a set such that $x_0\in \overline{X}$.
We say that $v\in \mathbb{K}^{m}$ is a {\bf tangent vector\index{tangent vector} of $X$ at $x_0\in\mathbb{K}^{m}$} if there are a sequence of points $\{x_i\}\subset X$ tending to $x_0$ and a sequence of positive real numbers $\{t_i\}$ such that 
$$\lim\limits_{i\to \infty} \frac{1}{t_i}(x_i-x_0)= v.$$
Let $C(X,x_0)$ denote the set of all tangent vectors of $X$ at $x_0$. We call $C(X,x_0)$ the {\bf tangent cone\index{tangent cone}\index{cone!tangent} of $X$ at $x_0$}. 
\end{definition}

Recall that a subset of $\C^n$ is called a {\bf complex cone}\index{cone!complex} 
if it is a union of one-dimensional complex linear subspaces of $\C^n$.

\begin{remark}\label{remark-tangent-cone}
{\rm In the case where $X\subset \C^m$ is a complex analytic set such that $0\in X$, $C(X,0)$ is the zero locus of a set of complex homogeneous polynomials (see \cite[Theorem 4D]{Whitney:1972}). In particular, $C(X,0)$ is a complex algebraic subset of $\C^m$ and is a complex cone. More precisely,  let $\mathcal{I}(X)$ be the ideal of $\mathcal{O}_{\C^m,0}$ given by the germs which vanishes on $X$. For each $f\in\mathcal{O}_{\C^m,0}$, let
$$ f = f_k+f_{k+1}+\cdots $$ be its Taylor development where each $f_j$ is a homogeneous polynomial\index{polynomial!homogeneous} of degree $j$ and $f_k\neq 0$. So, we say that $f_k$ is the {\bf initial part}\index{initial part} of $f$ and we denote it by ${\bf in}(f)$. In this way, $C(X,0)$ is the affine variety of the ideal $\mathcal{I}_*(X)=\langle{\bf in}(f);\, f\in \mathcal{I}(X)\rangle $}.
\end{remark}

\begin{example}
Examples of tangent cones\index{tangent cones}:
\begin{itemize}
 \item [(1)] If $X=\{(x,y)\in \R^2; x^3=y^2\}$ then $C(X,0)=\{(x,y)\in \R^2; y=0$ and $x\geq 0\}$;
 \item [(2)] Let $f\colon (\C^n,0)\to (\C,0)$ be a complex analytic function and let $ f = f_k+f_{k+1}+\cdots $ be its Taylor development at the origin where each $f_j$ is a homogeneous polynomial of degree $j$ and $f_k\neq 0$. Then $C(V(f),0)=V(f_k)$;
 \item [(3)] If $X=\{(x,y)\in \C^2; x^3=y^2\}$ then $C(X,0)=\{(x,y)\in \C^2; y=0\}$.
\end{itemize}
\end{example}

\begin{proposition}\label{tg:prop:basic_prop_tg}
Basic properties:
\begin{itemize}
 \item [(1)] If $A\subset X$ then $C(A,p)\subset C(X,p)$ for all $p\in \overline{A}$;
 \item [(2)] For $X,Y\subset \mathbb{R}^m$ and $p\in \overline{X}\cap \overline{Y}$, we have
 \begin{itemize}
  \item [(a)] $C(X\cup Y,p)=C(X,p)\cup C(Y,p)$;
  \item [(b)] $C(X\cap Y,p)\subset C(X,p)\cap C(Y,p)$;
 \end{itemize}
 \item [(3)] For $X\subset \mathbb{R}^m$, $Y\subset \mathbb{R}^n$ and $(p,q)\in \overline{X}\times \overline{Y}$, we have $C(X\times Y,(p,q))=C(X,p)\times C(Y,p)$;
 \item [(4)] If $X$ is a $C^1$-smooth submanifold of $\R^m$ then $C(X,p)=T_pX$ for all $p\in X$, where $T_pX$ is the tangent space\index{tangent space} of $X$ at $p$;
 \item [(5)] Let $\varphi\colon (\R^n,p)\to (\R^m,\varphi(p))$ be the germ of a mapping which is differentiable at $p$. If $X\subset \mathbb{R}^m$ is a set such that $p\in \overline{X}$ then $D\varphi_p(C(X,p))\subset C(\varphi(Y), \varphi(p))$, where $D\varphi_p$ is the derivative of $\varphi$ at $p$;
 \item [(6)] If $X\subset \R^m$ and $p\in \overline{X}$ then $C(X,p)$ is a closed subset of $\mathbb{R}^m$, $C(X,p)=C(\overline{X},p)$ and $C(X,p)$ is a real cone\index{cone!real}, i.e., for each $v\in C(X,p)$, we have that $\lambda v\in C(X,p)$ for all $\lambda>0$.
\end{itemize} 
\end{proposition}
\begin{proof}
The proof is left as an exercise for the reader.
\end{proof}

Note that the tangent cones\index{tangent cones!are not topological invariant} are not topological invariant in the sense that $(X,p)$ can be topological equivalent to $(Y,q)$, but $C(X,p)$ and $C(Y,q)$ are not homeomorphic . Indeed, the real cusp $X=\{(x,y)\in \R^2; y^3=x^2\}$ is homeomorphic to the real line $L=\{(x,y)\in \R^2; y=0\}$, but $C(X,0)=\{(0,y)\in \R^2; y\geq 0\}$ is not homeomorphic to $C(L,p)=L$, for any $p\in L$. The main result of this section is to prove that the tangent cones are bi-Lipschitz invariant\index{tangent cones!are bi-Lipschitz invariant}.

Another way to present the tangent cone of a subset $X\subset\R^m$ at the origin $0\in\R^m$ is via the spherical blow-up of $\R^m$ at the point $0$ as it is going to be done in the following: let us consider the {\bf spherical blowing-up}\index{spherical blowing-up} at the origin of $\R^m$
$$
\begin{array}{ccl}
\rho_m\colon\mathbb{S}^{m-1}\times [0,+\infty )&\longrightarrow & \R^m\\
(x,r)&\longmapsto &rx.
\end{array}
$$

Notice that $\rho_m\colon\mathbb{S}^{m-1}\times (0,+\infty )\to \R^m\setminus \{0\}$ is a homeomorphism with inverse map\index{map!inverse} $\rho_m^{-1}\colon\R^m\setminus \{0\}\to \mathbb{S}^{m-1}\times (0,+\infty )$ given by $\rho_m^{-1}(x)=(\frac{x}{\|x\|},\|x\|)$. Let us denote
$$X':=\overline{\rho_m^{-1}(X\setminus \{0\})} \mbox{ and } \partial X':=X'\cap (\mathbb{S}^{m-1}\times \{0\}).$$

\begin{proposition}\label{tg:exer:link_tg__divisor_bs}
If $X\subset \R^m$ is a subanalytic set and $0\in X$, then $\partial X'=\mathbb{S}_0X\times \{0\}$, where $\mathbb{S}_0X=C(X,0)\cap \mathbb{S}^{m-1}$.
\end{proposition}
\begin{proof}
The proof is left as an exercise for the reader.
\end{proof}

\subsection{Characterization of tangent cones}

In this section, we present one more way to consider tangent vectors of subanalytic sets. More precisely, the following proposition and its corollaries give us nice characterizations\index{characterization} of tangent vector of subanalytic sets $X$ in terms of velocity of arcs in $X$. 

We start by recalling the following fundamental result in real algebraic geometry (see \cite[Lemma 6.3]{BierstoneM:1988}):
\begin{lemma}[Curve Selection Lemma]\label{selection_curve}\index{Curve Selection Lemma}
Let $X$ be a subanalytic subset of $\R^n$ and $x\in\R^n$ being a non-isolated point of $\overline{X}$. Then, there
exist $\delta>0$ and an analytic map $\gamma\colon (-\delta,\delta)\to \R^n$ such that $\gamma(0)=x$ and $\gamma((0,\delta))\subset X$.
\end{lemma}

\begin{proposition}\label{selection_lemma_cone}
Let $Z\subset \R^m$ be a subanalytic set with $p\in \overline{Z\setminus \{p\}}$. A vector $v\in\R^m$ is a tangent vector\index{tangent vector} of $Z$ at $p$ if and only if there exists a continuous subanalytic arc $\gamma\colon [0,\varepsilon )\to \overline{Z}$ such that $\gamma((0,\varepsilon))\subset Z$ and $\gamma(t)-p=tv+o(t),$ where $g(t)=o(t)$ means $\lim\limits _{t\to 0^+ }\frac{g(t)}{t}=0$. 
\end{proposition}
\begin{proof}
Without loss of generality, we assume $p=0$. Thus, $Y=\rho_m^{-1}(X\setminus \{0\})\subset \mathbb{S}^{m-1}\times [0,+\infty )$ and $\overline{Y}$ are subanalytic sets. We are going to consider two cases:

\bigskip

\noindent 1) Case $v\not=0$. Since $v$ is a tangent vector of $Z$ at $0$, there are a sequence $\{s_k\}_{k\in \N}$ of positive real numbers and a sequence $\{z_k\}_{k\in \N}\subset Z$ such that $\lim\limits _{k\to +\infty }z_k=0 $ and $\lim\limits _{k\to +\infty }\frac{1}{s_k}z_k=v$. In particular, $\lim\limits _{k\to \infty } \frac{z_k}{\|z_k\|}=\frac{v}{\|v\|}$ and $u=(\frac{v}{\|v\|},0)\in \overline{Y}$. Then by Curve Selection Lemma (Lemma \ref{selection_curve})\index{Curve Selection Lemma}, there exists an analytic arc $\beta\colon (-\delta,\delta)\to \mathbb{S}^{n-1}\times \R$ such that $\beta(0)=u$ and $\beta((0,\delta ))\subset Y$. By writing $\beta(t)=(x(t),s(t))$, we have that $s\colon [0,\delta )\to \R$ is an analytic and non-constant function such that $s(0)=0$ and $s(t)>0$ if $t\in (0,\delta )$. By analyticity of $s'$, one can suppose that $s$ is strictly increasing in the domain $[0,\delta)$. Hence, $s\colon [0,\delta/2]\to [0,\delta' ]$ is a subanalytic homeomorphism, where $\delta' =s(\frac{\delta}{2})$. We define $\gamma\colon [0,\varepsilon)\to \overline{Z}$ by
$$
\gamma(t)=\rho_m\circ \beta\circ s^{-1}(t\|v\|)=\rho_m(x(s^{-1}(t\|v\|)),s(s^{-1}(t\|v\|)))=t\|v\|x(s^{-1}(t\|v\|)),
$$
where $\varepsilon=\min\{\frac{\delta'}{\|v\|},\delta'\}$. Therefore,
$$
\lim\limits _{t\to 0^+}\frac{\gamma(t)}{t}=\lim\limits _{t\to 0^+}\frac{t\|v\|x(s^{-1}(\|v\|t))}{t}=\lim\limits _{t\to 0^+}\|v\|x(s^{-1}(\|v\|t))=\|v\|x(0)=v,
$$
and thus $\gamma(t)=tv+o(t)$ and $\gamma((0,\varepsilon))\subset Z$. 
Since $\gamma$ is a composition of proper continuous subanalytic maps, it is a continuous subanalytic map as well.

\bigskip

\noindent 2) Case $v=0$. In this case,  let $\{z_k\}_{k\in\N}\subset Z$ be a sequence such that  $\lim\limits _{k\to +\infty }z_k=0 $. Thus, $\{\frac{x_k}{\|x_k\|}\}_{k\in\N}$ is, up to take subsequence, a convergent sequence. Let $v\in  \R^m$ be the limit of this sequence, i.e., $\lim\limits _{k\to \infty }\frac{x_k}{\|x_k\|}=v$. Likewise as it was done in the Case 1), one can show that there exists a continuous subanalytic arc $\gamma\colon [0,\varepsilon)\to \overline{Z}$ such that $\gamma(t)=tv+o(t)$. Let us define $\widetilde\gamma\colon [0,\varepsilon^{\frac{1}{2}})\to \overline{Z}$ by $\widetilde{\gamma}(t)=\gamma(t^{2})$. Thus, we have $\widetilde{\gamma}(t)=o(t)=tv+o(t)$.

\bigskip

Reciprocally, if there exists a continuous subanalytic arc $\gamma\colon [0,\varepsilon)\to \overline{Z}$ such that $\gamma(t)=tv+o(t)$ and $\gamma((0,\varepsilon))\subset Z$, then for each $k\in \N$ we define $s_k=\frac{\varepsilon}{k+2}$ and $z_k=\gamma(s_k)$. Thus, it is clear that $v$ is a tangent vector of $Z$ at $0$, since $\lim\limits _{k\to +\infty } z_k=0 $ and $\lim\limits _{k\to +\infty }\frac{1}{s_k}z_k=v$.
\end{proof}

It is an immediate consequence of Proposition \ref{selection_lemma_cone} the following result:

\begin{corollary}
Let $X$ be a subanalytic set in $\R^m$, $p \in \overline{X\setminus \{p\}}$. Then 
$C(X,p)=\{v\in R^m;$ there exists a continuous subanalytic arc $\gamma\colon [0,\varepsilon )\to \overline{Z}$ such that $\gamma((0,\varepsilon))\subset Z$ and $\gamma(t)-p=tv+o(t)\}$.
\end{corollary}

\begin{proposition}\label{tg:exer:norm_tg_curve}
 Let $X$ be a subanalytic set in $\R^m$, $p \in \overline{X}$. If $v\in C(X,p)\setminus \{0\}$ then we can take a continuous subanalytic arc $\gamma\colon [0,\varepsilon )\to \overline{X}$ such that $\gamma((0,\varepsilon))\subset X$, $\gamma(t)-p=tv+o(t)$ and, moreover, satisfies anyone of the following conditions: 
 \begin{itemize}
  \item [(i)] $\|\gamma(t)-p\|<t\|v\|$; 
  \item [(ii)] $\|\gamma(t)-p\|=t\|v\|$; 
  \item [(iii)] $\|\gamma(t)-p\|>t\|v\|$, for all $t\in (0,\varepsilon)$.
 \end{itemize}
\end{proposition}
\begin{proof}
The proof is left as an exercise for the reader.
\end{proof}

\subsection{Bi-Lipschitz invariance of tangent cones}

The main goal of this section is to show the so-called bi-Lipschitz invariance of the tangent cones of subanalytic sets in the following sense: if $(X,p)$ is bi-Lipschitz equivalent to $(Y,q)$, then $C(X,p)$ is bi-Lipschitz homeomorphic to $C(Y,q)$.

\begin{lemma}[McShane-Whitney-Kirszbraun's Theorem \cite{Mcshane:1934}, \cite{Whitney:1934} and \cite{Kirszbraun:1934}]\label{mcshane}\index{Theorem!McShane-Whitney-Kirszbraun's}
Let $h\colon X\subset \R^n\to \R^m$ be a Lipschitz mapping. Then there exists a Lipschitz mapping $H\colon \R^n\to \R^m$ such that $H|_X=h$.
\end{lemma}
\begin{proof}
It is enough to consider the case that $X$ is a closed subset and $m=1$.
Let $C>0$ be a constant such that $|h(x)-h(y)|\leq C\|x-y\|$ for all $x,y\in X$. We define $H\colon \R^n\to \R$ by $H(x)=\inf \{h(y)+C\|x-y\|;y\in X\}$. For $x\in X$, we have that $h(x)-h(y)\leq \|h(x)-h(y)\|\leq C\|x-y\|$, and thus $H(x)=h(x)$. Given $u,v\in \R^n$, for each $\varepsilon>0$, let $x_0,y_0\in X$ such that $H(u)>h(x_0)+C\|x_0-u\|-\varepsilon$ and $H(v)>h(x_0)+C\|x_0-v\|-\varepsilon$. Thus,
$$
H(u)-H(v)\leq h(y_0)+C\|y_0-u\|-h(y_0)-C\|y_0-v\|+\varepsilon\leq C\|u-v\|+\varepsilon
$$
and 
$$
H(v)-H(u)\leq h(x_0)+C\|x_0-v\|-h(x_0)-C\|x_0-u\|+\varepsilon\leq C\|u-v\|+\varepsilon.
$$
Finally, by taking $\varepsilon\to 0^+$, we get $|H(u)-H(v)|\leq C\|u-v\|$, which finishes the proof.
\end{proof}

Now we can state and prove the main result of this section. 

\begin{theorem}[Theorem of the bi-Lipschitz invariance of the tangent cones]\label{tg:thm:equivcone}\index{Theorem!of the bi-Lipschitz invariance of the tangent cones}
Let $X\subset \R^n$ and $Y\subset \R^m$ be subanalytic sets. If there are constants $C_1,C_2>0$ and a bi-Lipschitz homeomorphism $\phi\colon (X,x_0)\to (Y,y_0)$ such that 
$$
\frac{1}{C_1}\|x-y\|\leq \|\phi(x)-\phi(y)\|\leq C_2\|x-y\|, \quad \forall x,y \in X,
$$ 
then there is a global bi-Lipschitz homeomorphism\index{homeomorphism!global bi-Lipschitz} $d\phi\colon C(X,x_0)\to C(Y,y_0)$ such that $d\phi(0)=0$ and 
$$
\frac{1}{C_1}\|x-y\|\leq \|d\phi(x)-d\phi(y)\|\leq C_2\|x-y\|, \quad \forall x,y \in C(X,x_0).
$$
\end{theorem}

\begin{proof}
This proof follows closely the proof presented in \cite{Sampaio:2016}. The last part of this proof follows from the ideas presented in the proof of Theorem 3.1 in \cite{Sampaio:2019}.

Let $\phi: X\to Y$ be a bi-Lipschitz homeomorphism. By McShane-Whitney-Kirszbraun's Theorem\index{Theorem!McShane-Whitney-Kirszbraun's}, there exists $\widetilde{\phi}: \R^n\to \R^m$ a Lipschitz map such that $\widetilde{\phi}|_X=\phi$ and $\widetilde{\psi}: \R^m\to \R^n$ another Lipschitz map such that $\widetilde{\psi}|_Y=\phi^{-1}$.
Let us define $\varphi, \psi: \R^n\times \R^m\to \R^n\times \R^m$ as follows:
$$
\varphi(x,y)=(x-\widetilde{\psi}(y+\widetilde{\phi}(x)),y+\widetilde{\phi}(x))
$$
and
$$
\psi(z,w)=(z+\widetilde{\psi}(w), w-\widetilde{\phi}(z+\widetilde{\psi}(w))).
$$
Since $\varphi$ and $\psi$ are composition of Lipschitz maps\index{maps!composition of Lipschitz}, they are also Lipschitz maps. 

Next, we show that $\psi=\varphi^{-1}$.  In fact, if $(x,y)\in \R^n\times \R^m$ then

\begin{eqnarray*}
\psi(\varphi(x,y))&=& \psi(x-\widetilde{\psi}(y+\widetilde{\phi}(x)),y+\widetilde{\phi}(x))\\
				  &=& (x-\widetilde{\psi}(y+\widetilde{\phi}(x))+\widetilde{\psi}(y+\widetilde{\phi}(x)),y+\widetilde{\phi}(x)-\\
				  & &\quad\quad\quad\quad\quad\quad\quad\quad\quad\quad\quad\quad\widetilde{\phi}(x-\widetilde{\psi}(y+\widetilde{\phi}(x))+\widetilde{\psi}(y+\widetilde{\phi}(x)))\\
				  &=& (x,y+\widetilde{\phi}(x)-\widetilde{\phi}(x))\\
				  &=& (x,y),
\end{eqnarray*}
and if $(z,w)\in \R^n\times \R^m$ then
\begin{eqnarray*}
\varphi(\psi(z,w))&=&\varphi(z+\widetilde{\psi}(w), w-\widetilde{\phi}(z+\widetilde{\psi}(w)))\\
				  &=& (z+\widetilde{\psi}(w)-\widetilde{\psi}(w-\widetilde{\phi}(z+\widetilde{\psi}(w))+\widetilde{\phi}(z+\widetilde{\psi}(w))),w-\\
				  & &\quad\quad\quad\quad\quad\quad\quad\quad\quad\quad\quad\quad\quad\quad\widetilde{\phi}(z+\widetilde{\psi}(w))+\widetilde{\phi}(z+\widetilde{\psi}(w)))\\
				  &=& (z+\widetilde{\psi}(w)-\widetilde{\psi}(w),w)\\
				  &=& (z,w).
\end{eqnarray*}

Therefore $\psi=\varphi^{-1}$. Finally, it is clear that  $\varphi(X\times \{0\})=\{0\}\times Y$.

Thus, by doing the identifications $X\leftrightarrow X\times \{0\}$ and $Y\leftrightarrow \{0\}\times Y$, we may assume that $X$ and $Y$ are subsets of same $\R^N$ and there is a bi-Lipschitz homeomorphism\index{homeomorphism!bi-Lipschitz} $\varphi\colon \R^N\to \R^N$ such that $\varphi|_X=\phi$.

Without loss of generality, we assume that $x_0=y_0=0$. Let $K>0$ be a constant such that
\begin{equation}
\frac{1}{K}\|x-y\|\leq \|\varphi(x)-\varphi(y)\|\leq K\|x-y\|, \quad \forall x,y\in \R^N.
\end{equation}

For each $k\in\N$, let us define the maps $\varphi_k,\psi_k:\R^N\to \R^N$ given by $\varphi_k(v)=k\varphi(\frac{1}{k}v)$ and $\psi_k(v)=k\varphi^{-1}(\frac{1}{k}v)$. For each integer $m\geq 1$, let us define $\varphi_{k,m}:=\varphi_k|_{\overline{B}_m}:\overline{B}_m\to \R^N$ and $\psi_{k,m}:=\psi_k|_{\overline{B}_{mK}} :\overline{B}_{mK}\to \R^N$, where $\overline{B}_r$ denotes the Euclidean closed ball\index{ball!Euclidean closed} of radius $r$ and with centre at the origin in $\R^N$.
Since
$$
\frac{1}{K}\|x-y\|\leq \|\varphi_{k,1}(x)-\varphi_{k,1}(y)\|\leq K\|x-y\|, \quad \forall x,y\in \overline{B}_1,\,\, \forall k\in \N
$$
and
$$
\frac{1}{K}\|u-v\|\leq \|\psi_{k,1}(u)-\psi_{k,1}(v)\|\leq K\|u-v\|, \quad u,v\in \overline{B}_K,\,\, \forall k\in \N,
$$
there exist a subsequence $\{k_{j,1}\}_{j\in \N}\subset \N$ and Lipschitz maps $d\varphi_1:\overline{B}_1\to \R^N$ and $d\psi_1: \overline{B}_K\to \R^N$ such that $\varphi_{k_{j,1},1}\to d\varphi_1$ uniformly on $\overline{B}_1$ and $ \psi_{k_{j,1},1}\to d\psi_1$ uniformly on $\overline{B}_K$ (notice that $\{\varphi_{k,1}\}_{k\in \N}$ and $\{\psi_{k,1}\}_{k\in \N}$ have uniform Lipschitz constants)\index{constants!Lipschitz}. Furthermore, it is clear that
$$
\frac{1}{K}\|u-v\|\leq \|d\varphi_1(u)-d\varphi_1(v)\|\leq K\|u-v\|, \quad \forall u,v\in \overline{B}_1
$$
and
$$
\frac{1}{K}\|z-w\|\leq \|d\psi_1(z)-d\psi_1(w)\|\leq K\|z-w\|, \quad \forall z,w\in \overline{B}_K.
$$
Likewise as above, for each $m>1$, we have
$$
\frac{1}{K}\|x-y\|\leq \|\varphi_{k,m}(x)-\varphi_{k,m}(y)\|\leq K\|x-y\|, \quad x,y\in \overline{B}_m,\,\, \forall k\in \N
$$
and
$$
\frac{1}{K}\|u-v\|\leq \|\psi_{k,m}(u)-\psi_{k,m}(v)\|\leq K\|u-v\|, \quad u,v\in \overline{B}_{mK},\,\, \forall k\in \N.
$$
Therefore, for each $m>1$, there exist a subsequence $\{k_{j,m}\}_{j\in \N}\subset\{k_{j,m-1}\}_{j\in \N}$ and Lipschitz maps $d\varphi_m\colon \overline{B}_m\to \R^N$ and $d\psi_m\colon \overline{B}_{mK}\to \R^N$ such that $\varphi_{k_{j,m},m}\to d\varphi_m$ uniformly on $\overline{B}_m$ and $ \psi_{k_{j,m},m}\to d\psi_m$ uniformly on $\overline{B}_{mK}$ with $d\varphi_m|_{\overline{B}_{m-1}}=d\varphi_{m-1}$ and $d\psi_m|_{\overline{B}_{(m-1)K}}=d\psi_{m-1}$. Furthermore,
\begin{equation}\label{diflipum}
\frac{1}{K}\|u-v\|\leq \|d\varphi_m(u)-d\varphi_m(v)\|\leq K\|u-v\|, \quad \forall u,v\in \overline{B}_m
\end{equation}
and
\begin{equation}\label{diflipdois}
\frac{1}{K}\|z-w\|\leq \|d\psi_m(z)-d\psi_m(w)\|\leq K\|z-w\|, \quad \forall z,w\in \overline{B}_{mK}.
\end{equation}

Let us define $d\varphi,d\psi:\R^N\to \R^N$ by $d\varphi(x)=d\varphi_m(x)$, if $x\in \overline{B}_m$ and $d\psi(x)=d\psi_m(x)$, if $x\in \overline{B}_{mK}$ and, for each  $j\in \N$, let $n_j=k_{j,j}$ and $t_j=1/n_j$.

\begin{claim}\label{tg:claim:equivcone:one}
$\varphi_{n_j}\rightarrow d\varphi$ and $\psi_{n_j}\rightarrow d\psi$ uniformly on compact subsets of $\R^N$. 
\end{claim}
\begin{proof}
Let $F\subset \R^N$ be a compact subset. Let us take  $m\in \N$ such that $F\subset \overline{B}_m\subset \overline{B}_{mK}$. Thus, $\{n_j\}_{j>m}$ is a subsequence of $\{k_{j,m}\}_{j\in \N}$ and, since  $\varphi_{k_{j,m},m}\to d\varphi_m$ uniformly on $\overline{B}_m$ and $ \psi_{k_{j,m},m}\to d\psi_m$ uniformly on $\overline{B}_{mK}$, it follows that $\varphi_{n_j}\to d\varphi$ and $\psi_{n_j}\to d\psi$ uniformly on $F$.
\end{proof}

\begin{claim}\label{tg:claim:equivcone:two}
$d\varphi:\R^N\to \R^N$ is a bi-Lipschitz homeomorphism\index{homeomorphism!bi-Lipschitz} and $d\psi= (d\varphi)^{-1}$. 
\end{claim}
\begin{proof}
It follows from inequalities (\ref{diflipum}) and (\ref{diflipdois}) that $d\varphi,d\psi:\R^N\to \R^N$ are Lipschitz mappings. Therefore, it is enough to show that $d\psi= (d\varphi)^{-1}$. In order to do that, let $v\in \R^N$ and $w=d\varphi(v)=\lim \limits_{j\to \infty }\frac{\varphi(t_jv)}{t_j}$. Thus,
$$
\begin{array}{lllll}
\|d\psi(w)-v\|&=&\bigg\|\lim \limits_{j\to \infty }\frac{\psi(t_jw)}{t_j} - v\bigg\| &=&\lim \limits_{j\to \infty }\bigg\|\frac{\psi(t_jw)}{t_j} - \frac{t_jv}{t_j}\bigg\|\\
				    &=&\lim \limits_{j\to \infty }\frac{1}{t_j}\bigg\|\psi(t_jw) - t_jv\bigg\| &=&\lim \limits_{j\to \infty }\frac{1}{t_j}\bigg\|\psi(t_jw) - \psi(\varphi(t_jv))\bigg\|\\
				    &\leq &\lim \limits_{j\to \infty }\frac{K}{t_j}\bigg\|t_jw - \varphi(t_jv)\bigg\| &= &\lim \limits_{j\to \infty }K\bigg\|w - \frac{\varphi(t_jv)}{t_j}\bigg\|\\
					&=&0.& &
\end{array}
$$

Then, $d\psi(w)=d\psi(d\varphi(v))=v$, for all $v\in \R^N$, i.e., $d\psi\circ d\varphi={\rm id}_{\R^N}$. Analogously, one can show that $d\varphi\circ d\psi={\rm id}_{\R^N}$.
\end{proof}

\begin{claim}\label{tg:claim:equivcone:three}
 $d\varphi(C(X,0))= C(Y,0)$.
\end{claim}
\begin{proof}

By the previous claim, it is enough to verify that $d\varphi(C(X,0))\subset C(Y,0)$. In order to do that, let  $v\in C(X,0)$. Then, there is $\alpha\colon [0,\varepsilon)\to X$ such that $\alpha(t)=tv+o(t)$. Thus, $\varphi(\alpha(t))=\varphi(tv)+o(t)$, since $\varphi$ is a Lipschitz map. However, $\varphi(t_jv)=t_jd\varphi(v)+o(t_j)$ and then
$$
d\varphi (v)=\lim \limits_{j\to \infty }\varphi_{n_j}(v)=\lim \limits_{j\to \infty }\frac{\varphi(t_jv)}{t_j}=\lim \limits_{j\to \infty }\frac{\varphi(\alpha(t_j))}{t_j}\in C(Y,0).
$$
\end{proof}

Therefore, $d\varphi\colon C(X,0)\to C(Y,0)$ is a bi-Lipschitz homeomorphism.

\begin{claim}\label{tg:claim:equivcone:four}
 $\frac{1}{C_1}\|v-w\|\leq \|d\varphi(v)-d\varphi(w)\|\leq C_2\|v-w\|, \quad \forall v,w \in C(X,0).$
\end{claim}
\begin{proof}
Let $v\in C(X,0)$. By Proposition \ref{selection_lemma_cone}, there is a curve $\gamma\colon [0,\varepsilon )\to X$  such that $\gamma(t)=tv+o(t)$. Then, we obtain 
$$
\textstyle{\left\|\frac{\varphi(t_jv)}{t_j}-\frac{\varphi(\gamma(t_j))}{t_j}\right\|}=\frac{o(t_j)}{t_j}\to 0 \mbox{ as } j\to +\infty .
$$
Therefore, 
$$\textstyle{\lim\limits _{j\to +\infty}\frac{\varphi(t_jv)}{t_j}=\lim\limits _{j\to +\infty}\frac{\varphi(\gamma(t_j))}{t_j}=d\varphi(v).}$$ 
As $\varphi|_{X\times \{0\}}=0\times \phi$, we have 
\begin{equation} \label{eq_lim}
\textstyle{\lim\limits _{j\to +\infty}\frac{\phi(\gamma(t_j))}{t_j}= d\varphi(v).}
\end{equation} 
Therefore, if $v,w\in C(X,0)$, there are curves $\gamma, \beta\colon [0,\varepsilon )\to X$  such that  $\gamma(t)=tv+o(t)$ and $\beta(t)=tw+o(t)$. Thus, by the hypothesis of the theorem, we get
$$
\textstyle{\frac{1}{C_1}\left\|\frac{\gamma(t_j)}{t_j}-\frac{\beta(t_j)}{t_j}\right\|\leq \left\|\frac{\phi(\gamma(t_j))}{t_j}-\frac{\phi(\beta(t_j))}{t_j}\right\|\leq C_2\left\|\frac{\gamma(t_j)}{t_j}-\frac{\beta(t_j)}{t_j}\right\|}.
$$
Passing to the limit $j\to +\infty $ and using (\ref{eq_lim}), we obtain
$$
\frac{1}{C_1}\|v-w\|\leq \|d\varphi(v)-d\varphi(w)\|\leq C_2\|v-w\|.
$$
\end{proof}
This proves the theorem.
\end{proof}


\subsection{Exercises}\label{sec:exercises_tg_cones}

\begin{exercise}
Prove Proposition \ref{tg:prop:basic_prop_tg}.
\end{exercise}

\begin{exercise}
Prove Proposition \ref{tg:exer:link_tg__divisor_bs}.
\end{exercise}

\begin{exercise}
Prove Proposition \ref{tg:exer:norm_tg_curve}.
\end{exercise}

\begin{exercise}\label{tg:exer:inv_tg_cones}
Let $X_1\subset \R^{m_1}$ and $X_2\subset \R^{m_2}$ be closed sets such that $C(X_i,p_i)=\{v\in \R^{m_i};$ there exists a continuous arc $\gamma\colon [0,\varepsilon )\to X_i$ such that $\gamma(t)-p_i=tv+o(t)\}$. If there is a bi-Lipschitz homeomorphism $\phi\colon (X_1,p_1)\to (X_2,p_2)$, then prove that there is a global bi-Lipschitz homeomorphism\index{homeomorphism!global bi-Lipschitz} $d\phi\colon C(X_1,p_1)\to C(X_2,p_2)$ such that $d\phi(0)=0$. 
\end{exercise}

\section[Lipschitz Regularity Theorem]{Lipschitz Regularity Theorem and the bi-Lipschitz invariance of the multiplicity 1}\label{sec:Multiplicity_one}
We start this section with the introduction of different notions of Lipschitz regular points\index{point!Lipschitz regular} (or Lipschitz regularity). The aim of this section is to show that, for any analytic subset $X\subset\C^n$, Lipschitz regularity of $X$ at $p\in X$ implies that $p$ is an analytic regular point\index{point!analytic regular} of $X$ (for any notion of Lipschitz regularity\index{Lipschitz regularity} we work with). 

\subsection{Notions of Lipschitz regularity}\label{subsec:lip_regularity_notions}
\begin{definition}\label{def:Lip_reg}
We say that a set $X\subset \R^n$ is {\bf Lipschitz regular}\index{Lipschitz regular} (resp. a {\bf Lipschitz submanifold})\index{Lipschitz submanifold} at $p\in X$ if there exist an open neighbourhood $U\subset \R^n$ and a bi-Lipschitz homeomorphism $\varphi\colon U\cap X \to B^d$ (resp. $\varphi\colon U \to B^n$ such that $\varphi(U\cap X)=B^n\cap (\R^d\times \{0\})$), where $B^k$ is the unit open ball\index{ball!open} of $\R^k$ centred at the origin.
\end{definition}

\begin{definition}\label{def:Lip_graph}
We say that a set $X\subset \R^n$ is a {\bf Lipschitz graph}\index{Lipschitz graph} at $p\in X$ if there exist an open neighbourhood $U\subset \R^n$ and a Lipschitz map $F\colon B^d \to \R^{n-d}$ such that $U\cap X=graph(F)$.
\end{definition}

It is clear the following:
\begin{itemize}
 \item [(1)] If $X$ is a Lipschitz graph at $p$ then it is a Lipschitz submanifold at $p$;
 \item [(2)] If $X$ is a Lipschitz submanifold at $p$ then it is Lipschitz regular at $p$.
\end{itemize}
However, the converses of (1) and (2) are not true in general for semialgebraic sets (see Exercises \ref{reg:exer:lip_reg_not_lip_submanifold} and \ref{reg:exer:lip_submanifold_not_Lip_graph}).



\subsection{$C^k$ smoothness of analytic sets}
In this subsection, we present other notions of regularity of sets apart from those presented in Subsection \ref{subsec:lip_regularity_notions}.
\begin{definition}\label{def:smoothness}
For $k\in \mathbb{N}\cup \{\infty, \omega\}$, we say that a set $X\subset \R^n$ is {\bf $C^k$ submanifold}\index{$C^k$ submanifold} or {\bf $C^k$ smooth}\index{smooth!$C^k$} at $p\in X$ if there exist an open neighbourhood $U\subset \R^n$ of $p$ and a $C^k$ diffeomorphism\index{diffeomorphism} (or a homeomorphism\index{homeomorphism} when $k=0$) $\varphi\colon U \to B^n$ such that $\varphi(U\cap X)=B^d\times \{0\}$ and $\varphi(p)=0$.
\end{definition}

Since we can see $\C^n$ as $\R^{2n}$, subsets of $\C^n$ are subsets of $\R^{2n}$. Therefore the notions of regularity or smoothness introduced in Definitions \ref{def:Lip_reg}, \ref{def:Lip_graph} and \ref{def:smoothness} make sense even to subsets of $\C^n$.

\begin{proposition}\label{C1-regularity} If a complex analytic subset $X$ is $C^1$ smooth\index{smooth!$C^1$} 
at $x_0\in X$, then $X$ is a complex analytic submanifold at $x_0$.
\end{proposition}
\begin{proof}
The proof is left as an exercise for the reader.
\end{proof}

\begin{example}\label{C0-regularity} The set $X=\{(x,y,z)\in \C^3; y^3=z^2$ and $x=0\}$ is $C^0$ smooth\index{smooth!$C^0$} 
at any $p\in X$. Indeed, let $\phi\colon \C\to Y=\{(x,y)\in \C^2; y^2=x^3\}$ be the homeomorphism given by $\phi(t)=(t^2,t^3)$.
Let $\psi \colon \C^2\to \C$ be a continuous extension of $\phi^{-1}$ and let $\varphi\colon \C\times \C^2\to \C^2\times \C$ be the map given by $\varphi(s,u)=(s+\psi(u),u-\phi(s+\psi(u)))$. 
We have that $\varphi$ is a homeomorphism such that $\varphi^{-1}(t,v)=(t-\psi(v+\phi(t)),v+\phi(t))$ and $\varphi(X)=\C\times \{0\}$.
However, since $m(X,0)=2$, by Proposition \ref{C1-regularity} and Corollary \ref{cor:regularity_vs_mult}, $X$ is not $C^1$ smooth at $0$.
\end{example}

\subsection{Lipschitz regularity of analytic sets}
The following result is due to Sampaio in \cite{Sampaio:2016}, but it was proved in a slight weaker version by Bibrair at al. in \cite{BirbrairFLS:2016}. 
\begin{theorem}[Lipschitz Regularity Theorem]\label{reg:thm:Lip_regularity} \index{Theorem!Lipschitz Regularity}
Let $X\subset\C^n$ be a complex analytic set. If $X$ is Lipschitz regular at $x_0\in X$, then $X$ is smooth at $x_0$.
\end{theorem}

\noindent{\bf \textit{Proof}} Let $X\subset\C^n$ be a $d$-dimensional complex analytic set. Assume that $X$ is Lipschitz regular at $x_0\in X$. 
Let $h\colon U\rightarrow B$ be 
a subanalytic bi-Lipschitz homeomorphism between an open neighbourhood $U$ of $x_0$ in 
$X$ and $B\subset\R^{2d}$, that is, an open Euclidean ball\index{ball!Euclidean} centred at the origin $0\in\R^{2d}$. Let us suppose that $h(x_0)=0$.  
By Theorem \ref{tg:thm:equivcone}, 
$dh\colon C(X,x_0) \rightarrow T_0B$ is a bi-Lipschitz homeomorphism between the tangent cones  $C(X,x_0)$  and $T_0B=\R^{2d}$. In particular, $C(X,x_0)$ is a topological manifold.

 The next result 
was proved by D. Prill in \cite{Prill:1967}:

\begin{theorem}[Prill's Theorem] \label{prill} \index{Theorem!Prill's}
Let $V\subset \C^n$ be a complex cone. If $0\in V$ has a neighborhood homeomorphic to a Euclidean ball\index{ball!Euclidean}, then $V$ is a linear subspace of $\C^n$.
\end{theorem} 

Now, since we consider complex analytic sets, the tangent cone at $x_0$ of a complex analytic set is
a complex cone (see Remark \ref{remark-tangent-cone}). Then $C(X,x_0)$ is a $d$-dimensional linear subspace of $\C^n$.

Moreover, by Proposition \ref{lne:exer:lne_lip_inv}, $X$ is LNE\index{LNE} at $x_0$. Thus, our main theorem is consequence of the following:

\begin{proposition}\label{proposition A}
Let $X\subset\C^n$ be a complex analytic subset. Let $x_0\in X$ be such that the tangent cone 
$C(X,x_0)$ is a linear subspace of $\C^n$. If there exists a neighbourhood $U$ of $x_0$ in $X$ such that $U$ is LNE\index{LNE}, then $X$ is smooth\index{smooth} at $x_0$.
\end{proposition}
\begin{proof} 
Since $E:=C(X,x_0)$ is a linear subspace  of $\C^n$, we can consider the orthogonal projection
$$P\colon \C^n\rightarrow E.$$
We may suppose that $x_0=0$ and $P(x_0)=0$. Let us choose linear coordinates $(x,y)$ in $\C^n$ such that $E=\{(x,y)\in\C^n;\,y=0\}.$

\begin{claim}\label{reg:claim:algebraicity_one_local}
There exist positive constants $C$ and $\rho$ such that $X\cap B_{\rho}\subset \{(x,y);\|y\|< C\|x\|\}$.
\end{claim}
\begin{proof}

Indeed, if this claim is not true, there exists a sequence $\{(x_k,y_k)\}_{k\in \N}\subset X$ such that $\lim\limits_{k\to+\infty}(x_k,y_k)=0$ and $\|y_k\|\geq k\|x_k\|$. Thus, up to a subsequence,  one can suppose that $\lim\limits_{k\to+\infty}\frac{y_k}{\|y_k\|}=y_0$. Since $\frac{\|x_k\|}{\|y_k\|}\leq \frac{1}{k}$, we obtain that  $(0,y_0)\in C(X,0)$, which is a contradiction, because $y_0\not=0$. Therefore, Claim \ref{reg:claim:algebraicity_one_local} is true.

\end{proof}
Notice that the germ of the restriction of the orthogonal projection $P$ to $X\cap B_{\rho}$ is a finite complex analytic map germ. 

Moreover, we have the following:

\begin{claim}\label{reg:claim:control_tangency_local}
If $\gamma\colon [0,\varepsilon)\rightarrow X$ is a real analytic arc, such that $\gamma(0)=0$, then the  arcs 
$\gamma$ and $P\circ\gamma$ are tangent at $0$. 
\end{claim}
\begin{proof}

In order to prove this claim, let us write $\gamma(t)=(x(t),y(t))$. By the previous claim, there exists $t_0>0$ such that  $\|y(t)\|\leq C\|x(t)\|$ for all $t\leq t_0$, since $\lim\limits_{t\to 0^+}\gamma(t)=0$.  Thus, since $\frac{x(t)}{t}$ is bounded, $\frac{y(t)}{t}$ is bounded. Let us suppose that $y(t)\not=o(t)$. Then, there exist a sequence $\{t_k\}_{k\in \N}\subset (0,+\infty)$ and $r>0$ such that $t_k\to 0$ and $\frac{\|y(t_k)\|}{t_k}\geq r$ for all $k$. Since $\left\{\frac{y(t_k)}{t_k}\right\}_{k\in \N}$ is bounded, up to a subsequence, one can suppose  that $\lim\limits_{k\to +\infty}\frac{y(t_k)}{t_k}=y_0$. Therefore, $\lim\limits_{k\to +\infty}\frac{\gamma(t_k)}{t_k}=(v',y_0)\in C(X,0)$, where $v=(v',0)$. However, this is a contradiction, since $\|y_0\|\geq r>0$ and this implies that $y_0\not=0$. Then, $y(t)=o(t)$ and, therefore, $\gamma(t)=tv+o(t)$.
\end{proof}

In this way, the germ at $0$ of $P_{| X}\colon X\rightarrow E$ is a ramified cover and the ramification locus is 
the germ of a codimension $\geq 1$ complex analytic subset $\Sigma$ of the linear space $E$. 

The multiplicity of $X$ at $0$ can be interpreted as the degree $m$ of this germ of ramified covering map, i.e. 
there are open neighbourhoods $U_1$ of $0$ in $X$ and $U_2$ of $0$ in $E$, such that $m$ is the
degree of the topological covering\index{topological covering}:
$$P_{| X}\colon X\cap U_1\setminus P_{| X}^{-1}(\Sigma)\rightarrow E\cap U_2\setminus \Sigma.$$ 
Let us suppose that the degree $m$ is greater than 1. Since $\Sigma$ is a codimension $\geq 1$ complex analytic subset of the space $E$, there exists a unit tangent vector $v_0\in E\setminus C(\Sigma,0)$. 

Since $v_0$ is not tangent to $\Sigma$ at $0$, there exists a positive real number $k$ such that the real cone: 
$$\{v\in E \ | \ \|v-tv_0\|< tk, \ \forall \  0<t<1\}$$ 
does not intersect the set $\Sigma$. Since we have assumed that the degree $m\geq 2$, we have at least two different liftings\index{liftings} $\gamma_1(t)$ and $\gamma_2(t)$ of the half-line $r(t)=tv_0$,  i.e. $P(\gamma_1(t))=P(\gamma_2(t))=tv_0$.   Since $P$ is the orthogonal projection on the tangent cone $E $, the vector $v_0$ is the unit tangent vector\index{tangent vector} to the arcs $\gamma_1$ and $\gamma_2$ at $0$. By
construction, we have ${\rm dist}(\gamma_i(t),P_{| X}^{-1}(\Sigma))\geq kt$ ($i=1,2$), where by dist we mean the Euclidean distance.

On the other hand, any path in $X$ connecting $\gamma_1(t)$ to $\gamma_2(t)$ is the lifting of a loop, based at the point $tv_0$, which is not contractible in the germ of  $E\setminus {\Sigma}$ at $0$. Thus the length of such a path must be at least  $2kt$. It implies that the inner distance, ${\rm d_{X}}(\gamma_1(t),\gamma_2(t))$, in $X$, between $\gamma_1(t)$ and $\gamma_2(t)$, is at least $2kt$. But, since $\gamma_1(t)$ and $\gamma_2(t)$ are tangent at $0$, that is $\displaystyle\frac{\|\gamma_1(t)-\gamma_2(t)\|}{t}\to 0 \ \mbox{as} \ t\to 0^+,$ and $k>0$, we obtain that $X$ is not LNE\index{LNE!not} near $0$. Otherwise there will be $\lambda >0$ such that:  
$$d_X(x_1,x_2)\le \lambda \|x_1-x_2\| \quad \mbox{for all} \quad  x_1,x_2\in X \quad \mbox{near} \ 0,$$ 
hence:
\begin{eqnarray*}
2k &\leq & \frac{\rm{d_{inner}}(\gamma_1(t),\gamma_2(t))}{t} \\
&\leq &\lambda \frac{\|\gamma_1(t)-\gamma_2(t)\|}{t}\to 0.
\end{eqnarray*}
which is a contradiction. 

Therefore, $m=m(X,0)=1$, and thus by Exercise \ref{mult:exer:mult_one_smooth} $X$ is smooth at $0$.
\end{proof}

This concludes the theorem. \hfill$\square$

As a consequence, we obtain that the multiplicity 1 is a bi-Lipschitz invariant.
\begin{corollary}\label{reg:cor2:Lip_regularity} 
Let $X\subset\C^n$ and $Y\subset\C^m$ be  complex analytic sets. If there is a bi-Lipschitz homeomorphism\index{homeomorphism!bi-Lipschitz} $\varphi\colon (X,0)\to (Y,0)$ then $m(X,0)=1$ if and only if $m(Y,0)=1$.
\end{corollary}

\subsection{Exercises}\label{sec:exercises_regularity}

\begin{exercise}\label{reg:exer:lip_reg_not_lip_submanifold}
Let $L=\{(x,y)\in\C^2; x^3=y^2\}\cap \mathbb{S}^3$. Prove that $X=Cone(L)=\{tu;u\in L$ and $t\geq 0\}$ is Lipschitz regular at $0$, but it is not a Lipschitz submanifold at $0$.
\end{exercise}

\begin{exercise}\label{reg:exer:lip_submanifold_not_Lip_graph}
Give an example of a semialgebraic set $X\subset \R^n$ which is Lipschitz submanifold at $0$ of $\R^n$, but it is not a Lipschitz graph at $0$.
\end{exercise}

\begin{exercise}
Let $X=\{(x,y,z)\in \R^3;x^3+y^3+z^3=0\}$. Prove that:
\begin{enumerate}
 \item [a)] $X$ is not $C^1$ smooth at $0$;
 \item [b)] $X$ is Lipschitz regular at $0$.
\end{enumerate}
\end{exercise}

\begin{exercise}\label{mult:exer:mult_one_smooth}
Let $X$ be a pure dimensional complex analytic set with $0\in X$. Prove that $X$ is smooth at $0$ if and only if $m(X,0)=1$.
\end{exercise}

\begin{exercise}\label{reg:cor:Lip_regularity} 
Let $X\subset\C^n$ and $Y\subset\C^m$ be  complex analytic sets. If there is a bi-Lipschitz homeomorphism $\varphi\colon X\to Y$ then prove that $\varphi(Reg(X))=Reg(Y)$ and $\varphi(Sing(X))=Sing(Y)$, where for a complex analytic set $A\subset \C^k$, $Reg(A)$ denotes the subset of points $x\in A$ such that for some open neighbourhood $U\subset \C^k$ of $x$, $X\cap U$ is a complex analytic submanifold of $\C^k$, and $Sing(A)=A\setminus Reg(A)$.
\end{exercise}

\section{Relative multiplicities and multiplicity of homogeneous singularities.}
In this section, we define the relative multiplicities of a complex analytic set at a point. Moreover, we prove that the relative multiplicities are bi-Lipschitz invariant (see Theorem \ref{mult:thm:inv_multiplicities}). We also present a reduction of the Questions AL($d$) and AL($n,d$) for homogeneous algebraic sets (see Corollary \ref{mult:cor:reduction_to_homogeneous}).

\subsection{Bi-Lipschitz invariance of the relative multiplicities}\label{subsection:invariance_rel_mult}

Let $X\subset \C^n$ be a complex analytic set such that $0\in X$. Let $X_1,...,X_r$ be the irreducible components of $C(X,0)$.
Fix $j\in \{1,...,r\}$. For a generic point\index{point!generic} $x\in (X_j\cap \mathbb{S}^{2n-1})\times \{0\}$, the number of connected components\index{components!number of connected} of the germ $(\rho^{-1}(X\setminus\{0\}),x)$ is constant, and we denote this number by $k_X(X_j)$. The integer numbers $k_X(X_j)$ are called the {\bf relative multiplities}\index{relative multiplicities} of $X$ at $0$.

The following result, which was proved by Fernandes and Sampaio in the paper \cite{FernandesS:2016}, shows the bi-Lipschitz invariance of the relative multiplicities\index{relative multiplicities!bi-Lipschitz invariance}.

\begin{theorem}\label{mult:thm:inv_multiplicities}
Let $X\subset \C^n$ and $Y\subset \C^m$ be two complex analytic sets with $p=\dim X=\dim Y$, $0\in X$ and $0\in Y$. Let $X_1,\dots,X_r$ and $Y_1,\dots,Y_s$ be the irreducible components\index{components!irreducible} of the tangent cones $C(X,0)$ and $C(Y,0)$, respectively. If there exists a bi-Lipschitz homeomorphism\index{homeomorphism!bi-Lipschitz} $\varphi\colon(X,0)\to (Y,0)$, then $r=s$ and, up to a reordering of indices,  $k_X(X_j)=k_Y(Y_j)$, $\forall \ j$.
\end{theorem}
\begin{proof} 
Let $S=\{t_k\}_{k\in\N}$ be a sequence of positive real numbers such that 
$$t_k\to 0 \quad\mbox{and} \quad \frac{\varphi(t_kv)}{t_k}\to d\varphi(v)$$ 
where $d\varphi$ is a tangent map of $\varphi$ as in Theorem \ref{tg:thm:equivcone}. Since $d\varphi$ is a bi-Lipschitz homeomorphism, we obtain $r=s$ and there is a permutation $\sigma\colon\{1,\dots,r\}\to \{1,\dots,s\}$ such that $d\varphi (X_i)=Y_{\sigma(i)}$ $\forall \ i.$ This is because we can assume $d\varphi(X_i)=Y_i$ $\forall \ i$ up to a reordering of indices. Let $$SX=\{(x,t)\in \mathbb{S}^{2n-1}\times S;\,tx \in X\}.$$ 
Thus, $\rho^{-1}\circ\varphi\circ\rho \colon SX\rightarrow Y'$ is an injective and continuous map that extends continuously to a map $\varphi'\colon \overline{SX}\to Y'.$
 
For each generic point $x\in \mathbb{S}_0 X_j\times\{0\}$, we know that $k_X(X_j)$ is the number of connected components of the set $\rho^{-1}(X\setminus \{0\})\cap B_{\delta }(x)$ for  $\delta>0$ small enough. Then, $k_X(X_j)$ can be seen as the number of connected components of the set $(SX\cap \mathbb{S}^{2n-1}\times \{t_k\})\cap B_{\delta }(x)$ for $k$ large enough.

Let $\pi\colon\C^n\to \C^p$ be a linear projection such that 
$$\pi^{-1}(0)\cap(C(X,0)\cup C(Y,0))=\{0\}.$$ 
Let us denote the ramification locus of 
$$\pi_{| X}\colon X\to \C^p \quad \mbox{and} \quad \pi_{| C(X,0)}\colon C(X,0)\to \C^p$$ 
by $\sigma(X)$ and $\sigma(C(X,0))$ respectively.

Given a generic point $v'\in \C^p\setminus (\sigma(X)\cup \sigma(C(X,0)))$ (generic here means that $v'$ defines a direction not tangent to $\sigma(X)\cup \sigma(C(X,0))$), let $\eta,\varepsilon >0$ be sufficiently small such that 
$$C_{\eta,\varepsilon }(v')=\{w\in \C^p|\, \exists t>0; \|tv'-w\|<\eta t\}\cap B_{\varepsilon }(0)\subset \C^p\setminus \sigma(X)\cup \sigma(C(X,0)).$$

The number of connected components of $\pi^{-1}(C_{\eta,\varepsilon }(v'))\cap X$ is exactly $m(X)$, since $C_{\eta,\varepsilon }(v')$ is simply connected and $\pi\colon  X\setminus \pi^{-1}(\sigma(X))\to \C^p\setminus\sigma(X)$ is a covering map. Then, we get that $\pi|_V\colon V\to C_{\eta,\varepsilon }(v')$ is bi-Lipschitz for each connected component $V$ of  $\pi^{-1}(C_{\eta,\varepsilon }(v'))\cap X$. Therefore, for each $j=1,\dots,r$, there are different connected components $V_{j1},\dots,V_{jk_X(X_j)}$ of $\pi^{-1}(C_{\eta,\varepsilon }(v'))\cap X$ such that $C(\overline{V_{ji}},0)\subset X_j$, $i=1,...,k_X(X_j)$.

Let us suppose that there is $j\in\{1,\dots,r\}$ such that $k_X(X_j)>k_Y(Y_j)$, it means that if we consider a generic point $x=(v,0)\in\partial X'\cap X_j\times\{0\}$, there are at least two different connected components $V_{ji}$ and $V_{jl}$ of $\pi^{-1}(C_{\eta,\varepsilon }(\pi(v)))\cap X$ and sequences $\{(x_k,t_k)\}_{k\in \N}\subset \rho^{-1}(V_{ji})\cap SX$ and $\{(y_k,t_k)\}_{k\in \N}\subset \rho^{-1}(V_{jl})\cap SX$ such that $\lim (x_k,t_k)=\lim (y_k,t_k)=x$ and $\varphi'(x_k,t_k),\varphi'(y_k,t_k)\in \rho^{-1}(\widetilde V_{jm})$, where $\widetilde V_{jm}$ is a connected component of  $\pi^{-1}(C_{\eta,\varepsilon }(\pi(d\varphi(v))))\cap Y$.

Since $\varphi(t_kx_k),\varphi(t_ky_k)\in\widetilde V_{jm}$ $\forall$ $k\in \N$ and $V=\widetilde V_{jm}$ is bi-Lipschitz homeomorphic to $C_{\eta,\varepsilon }(\pi(d\varphi(v)))$, we have 
$$\|\varphi(t_kx_k)-\varphi(t_ky_k)\|=o(t_k)$$ 
and
$$d_Y(\varphi(t_kx_k),\varphi(t_ky_k))\leq d_V(\varphi(t_kx_k),\varphi(t_ky_k))=o(t_k).$$
Now, since  $X$ is bi-Lipschitz homeomorphic to $Y$, we have $d_X(t_kx_k,t_ky_k)\leq o(t_k)$.
On the other hand, since $t_kx_k$ and $t_ky_k$ lie in different connected components of $\pi^{-1}(C_{\eta,\varepsilon }(\pi(v)))\cap X$, there exists a constant $C>0$ such that $d_X(t_kx_k,t_ky_k)\geq Ct_k$, which is a contradiction.

We have proved that $k_X(X_j)\leq k_Y(Y_j)$, $j=1,...,r$. By similar arguments, using that $\varphi^{-1}$ is a bi-Lipschitz map, we also can prove $k_Y(Y_j)\leq k_X(X_j)$, $j=1,...,r$.
\end{proof}

\subsection{Reduction to homogeneous algebraic sets}\label{subsection:reduction_to_homogeneous}

\begin{proposition}\label{mult:prop:kurdyka-raby}
Let $X\subset \C^n$ be a $d$-dimensional complex analytic set such that $0\in X$. Let $X_1,...,X_r$ be the irreducible components of $C(X,0)$. Then
$$
m(X,0)=\sum\limits_{j=1}^r k_X(X_j)m(X_j,0).
$$
\end{proposition}
\begin{proof}
The proof is left as an exercise for the reader.
\end{proof}

%
%
%

Next result follows from Theorem\ref{mult:thm:inv_multiplicities} and Proposition \ref{mult:prop:kurdyka-raby}.

\begin{theorem}\label{mult:thm:reduction_to_homogeneous}
Let $X\subset \C^n$ and $Y\subset \C^m$ be two complex analytic sets with $\dim X=\dim Y=d$, $0\in X$ and $0\in Y$. Let $X_1,...,X_r$ (resp. $Y_1,...,Y_s$) be the irreducible components\index{components!irreducible} of $C(X,0)$ (resp. $C(Y,0)$).
If there exists a bi-Lipschitz 
homeomorphism\index{homeomorphism!bi-Lipschitz} $\varphi\colon (X,0)\to (Y,0)$, then $r=s$ and there exist a bijection $\sigma\colon \{1,...,r\}\to \{1,...,s\}$ and a bi-Lipschitz homeomorphism $d\varphi\colon C(X,0)\to C(Y,0)$ with $d\varphi(0)=0$ and such that $d\varphi(X_i)=Y_{\sigma(i)}$ and $k_X(X_i)=k_Y(Y_{\sigma(i)})$ for all $i\in \{1,...,r\}$. Additionally, if $m(X_i,0)=m(Y_{\sigma(i)},0)$ for all $i\in \{1,...,r\}$, then $m(X,0)=m(Y,0)$.
\end{theorem}

As a direct consequence of the  above theorem, we obtain that to solve Question $AL(d)$ (resp. $AL(n,d)$), it is enough to solve it only for irreducible homogeneous algebraic sets.

\begin{corollary}\label{mult:cor:reduction_to_homogeneous}
Question $AL(d)$ (resp. $AL(n,d)$) has a positive answer if and only if it has a positive answer for irreducible homogeneous algebraic sets.
\end{corollary}

\subsection{Exercises}\label{sec:exercises_relative_multiplicities}

\begin{exercise}
Prove Proposition \ref{mult:prop:kurdyka-raby}.
\end{exercise}

\begin{exercise}
 Prove Theorem \ref{mult:thm:reduction_to_homogeneous}.
\end{exercise}


\section{Bi-Lipschitz invariance of the multiplicity}

This section is devoted to present a complete answer to the Question AL($d$) on the bi-Lipschitz invariance of the multiplicity. We also present some partial results of other authors concerning the questions AL($d$) and AL($n,d$). We start by addressing the case of  complex singularities of dimension 1.

\subsection{Multiplicity of curves}\label{subsec:mult_curves}

\begin{theorem}\label{mult:thm:mult_inv_curves}
Let $X\subset \C^{n}$ and $Y\subset \C^{m}$ be two complex analytic curves. If $\varphi\colon(X,0)\to (Y,0)$ is a bi-Lipschitz homeomorphism\index{homeomorphism!bi-Lipschitz}, then $m(X,0)=m(Y,0).$
\end{theorem}
\begin{proof}
Let us consider $X\subset \C^{n}$ and $Y\subset \C^{m}$ be two complex analytic curves; $0\in X$ and $0\in Y$. Then, we know that the tangent cones\index{tangent cones} $C(X,0)$ and $C(Y,0)$ are union of complex lines\index{tangent cones!complex lines} through the origin, let us say:
$$ C(X,0)=\bigcup_{i=1}^{r}L^{X}_i \ \mbox{and} \ C(Y,0)=\bigcup_{j=1}^{s} L^{Y}_j .$$
Since $\varphi\colon(X,0)\to (Y,0)$ is a bi-Lipschitz homeomorphism, it follows, by Theorem \ref{mult:thm:reduction_to_homogeneous}, that $r=s$ and there exist a bijection $\sigma\colon \{1,...,r\}\to \{1,...,s\}$ and a bi-Lipschitz homeomorphism $d\varphi\colon C(X,0)\to C(Y,0)$ with $d\varphi(0)=0$ and such that $d\varphi(L^{X}_i)=L^{Y}_{\sigma(i)}$ and $k_X(L^{X}_i)=k_Y(L^{Y}_{\sigma(i)})$ for all $i\in \{1,...,r\}$. Then, since $m(L^{Y}_j,0)=1$ and $m(L^{X}_i,0)=1$ $\forall i,j$, it follows from Theorem \ref{mult:thm:reduction_to_homogeneous} that $m(X,0)=m(Y,0)$.
\end{proof}

\subsection{Multiplicity of 2-dimensional analytic hypersurfaces}\label{subsec:mult_surfaces}

In this subsection, we bring a positive answer to the  Question AL(3,2). Notice that such a result is a consequence of the positive answer of Question AL(2) which the proof we sketch in the end of this section. Since our proof of Question AL(2) we sketch here relies on a non-trivial topological result that we do not prove here, we decide to keep this subsection with a proof of the positive answer for AL(3,2) which is quite complete and self-contained.

Let $f:\C^n \to \C$ be a homogeneous polynomial with ${\rm deg}{f}=d$. We recall that the map $\phi\colon\mathbb{S}^{2n-1}\setminus f^{-1}(0)\to \mathbb{S}^1$ given by $\phi(z)=\frac{f(z)}{|f(z)|}$ is a locally trivial fibration\index{fibration} (see \cite{Milnor:1968}, \S 4). Notice that, $\psi:\C^n\setminus f^{-1}(0)\to \C\setminus \{0\}$ defined by $\psi(z)=f(z)$ is a locally trivial fibration such that its fibres\index{fibres} are diffeomorphic\index{diffeomorphic} to the fibres of $\phi$.  Moreover, we can choose as geometric monodromy\index{monodromy!geometric} the homeomorphism $h_f:F_f\to F_f$ given by $h_f(z)=e^{\frac{2\pi i}{d}}\cdot z$, where $F_f:=f^{-1}(1)$ is the (global) Milnor fibre\index{fibre!Milnor} of $f$ (see \cite{Milnor:1968}, \S 9).

It follows from Theorem 3.3 and Remark 3.4 in \cite{Le:1973} the following:

\begin{proposition}\label{mult:prop:top_inv_milnor_fiber}
Let $f,g \subset \C^{n+1}\to \C$ be two irreducible homogeneous polynomials\index{polynomial!irreducible!homogeneous}. 
If $\varphi\colon(\C^{n+1},V(f),0)\to (\C^{n+1},V(g),0)$ is a homeomorphism, then the induced maps in homology of the monodromies\index{monodromy} of $f$ and $g$ at $0$ are conjugated. In particular, the monodromies of $f$ and $g$ at $0$ have the same Lefschetz number\index{number!Lefschetz} and $\chi(F_f)=\chi(F_g)$.
\end{proposition}

\begin{definition}
Let $f\colon (\C^{n+1},0)\to (\C,0)$ be a complex analytic function with 
$$\dim  Sing{V(f)}=1 \mbox{ and } Sing{V(f)}=C_1\cup...\cup C_r.$$ 
Then $b_i(f)$ denotes the $i$-th Betti number\index{number!Betti} of the Milnor fibre of $f$ at the origin, $\mu_j'(f)$ is the Milnor number\index{number!Milnor} of a generic hyperplane slice\index{slice} of $f$ at $x_j\in C_j\setminus \{0\}$ sufficiently close to the origin and $\mu'(f)=\sum\limits_{i=1}^r\mu_i'(f)$.
\end{definition}

Let us remind some results on Milnor number of a generic hyperplane slice.

In \cite{Massey:2017}, $\mu'(f)$ is denoted by $\sigma_f$ and it was proved in \cite[Theorem 4.1]{Massey:2017} that it is an embedded topological invariant\index{invariant!embedded topological}. We state that result here.
\begin{proposition}\label{mult:prop:top_inv_transv_milnor_number}
Let $f,g\colon (\C^{n+1},0)\to (\C,0)$ be complex analytic functions with $1$-dimensional singular sets\index{singular set!$1$-dimensional}.
If there is a homeomorphism\index{homeomorphism} $\varphi\colon(\C^{n+1},V(f),0)\to (\C^{n+1},V(g),0)$, then $\mu'(f)=\mu'(g)$.
\end{proposition}

\begin{proposition}[Theorem 5.11 in \cite{Randell:1979}]\label{mult:prop:Randell_formula}
Let $f\subset \C^{n+1}\to \C$ be a homogeneous polynomial with degree $d$ and $1$-dimensional singular set\index{singular set!$1$-dimensional}. 
Then 
$$\chi(F_f)=1+(-1)^n((d-1)^{n+1}- d\mu'(f)).$$
\end{proposition}


The following result was proved by Sampaio in \cite{Sampaio:2020}.
\begin{theorem}\label{mult:thm:mult_inv_hypersurface}
Let $X,Y \subset \C^{n+1}$ be two complex analytic hypersurfaces with $0\in X\cap Y$.  Assume that each irreducible component $X_i$ of $C(X,0)$ satisfies $\dim Sing{X_i}\leq 1$. If $\varphi\colon(\C^{n+1},X,0)\to (\C^{n+1},Y,0)$ is a bi-Lipschitz homeomorphism\index{homeomorphism!bi-Lipschitz}, then $m(X,0)=m(Y,0).$
\end{theorem}
\begin{proof} 
Let $\widetilde{f},\widetilde{g}\colon(\C^{n+1},0)\to (\C,0)$ be two reduced complex analytic functions such that $X=V(\widetilde{f})$ and $Y=V(\widetilde{g})$.
Let $f_1\cdots f_r$ (resp. $g_1,...,g_s$) be the irreducible factors of the decomposition of ${\bf in} (\widetilde{f})$ (resp. ${\bf in} (\widetilde{g})$) in irreducible polynomials\index{polynomial!irreducible}. Then, $r=s$ and by reordering the indices, if necessary, there exists a bi-Lipschitz homeomorphism\index{homeomorphism!bi-Lipschitz} $\psi=d\varphi\colon(\C^{n+1},0)\to (\C^{n+1},0)$ such that $\psi(V(f_i))=V(g_i)$ for all $i\in \{1,...,r\}$. Fixed $i$, we denote $f=f_i$, $g=g_i$, $d=m(V(f),0)$ and $e=m(V(g),0)$. By Proposition \ref{mult:prop:top_inv_milnor_fiber}, $\chi(F_f)=\chi(F_g)$.

We have two cases to consider:

\noindent Case 1) $\chi(F_f)\neq 0$, $\chi(F_g)\neq 0$ as well.

\begin{claim}\label{claim:mult_char_fiber}
 If $0<k<d$ (respectively $0<k<e$), then $\Lambda(h_f^k)=0$ (respectively $\Lambda(h_g^k)=0$), where $\Lambda(h_f^k)$ (respectively $\Lambda(h_g^k)$) denotes the Lefschetz number\index{number!Lefschetz} of $h_f^k$ (respectively $h_g^k$).
\end{claim}

\bigskip

\begin{proof}
We start the proof using the  Topological  Cylindric Structure\index{structure!cylindric} at Infinity of Algebraic Sets (see \cite{Dimca:1992}, p. 26, Theorem 6.9) to justify that $F=f^{-1}(1)$ has the same homotopy type of $F_R=F\cap \{x\in \C^n;\, \|x\|\leq R\}$ for $R$ large enough. We have that the geometric monodromy  $h_f\colon F\to F$ given by $h_f(x)=e^{\frac{2\pi i}{d}}x$, restricted to $F_R$, induces a map $h=h_f|_{F_R}\colon F_R\to F_R$. It is clear that $h^k$ does not have a fixed point for $0<k<d$, hence $\Lambda(h^k)=0$. Since $h_f$ is homotopy equivalent to $h$,  it follows that $\Lambda(h_f^k)=0$ for any $0<k<d$. 
\end{proof}

\bigskip

It follows from Proposition \ref{mult:prop:top_inv_milnor_fiber} that $\Lambda(h_f^k)=\Lambda(h_g^k)$ for all $k\in \N$. Since $f$ and $g$ are homogeneous polynomials with degrees $d$ and $e$ respectively, $h_f^d=id \colon  F_f\rightarrow F_f$ and $h_g^e=id \colon  F_g\rightarrow F_g$, we get $\Lambda(h_f^d)=\chi(F_f)\neq 0$ and $\Lambda(h_g^e)=\chi(F_g)\neq 0$. Thus, it follows from previous claim that $d=e$.

\noindent Case 2) $\chi (F_f)=\chi (F_g)=0$.

By the Lipschitz Regularity Theorem\index{Theorem!Lipschitz Regularity} (Theorem \ref{reg:thm:Lip_regularity}), we have that $\psi( Sing{(V(f))})= Sing{(V(g))}$ and, in particular, 
$\dim  Sing{(V(f))}=\dim  Sing{(V(g))}$. Thus, we can suppose that $d,k>1$.  If $\dim  Sing{(V(f))}=0$ then $\chi(F_f)=1+(-1)^n(d-1)^{n+1}=0$ and $\chi(F_g)=1+(-1)^n(k-1)^{n+1}=0$. This implies $d=k=2$. 

Thus, we can assume that $\dim  Sing{(V(f))}\not=0$. In this case, we have $\dim  Sing{V(f)}=\dim  Sing{(V(g))}=1$. Since $\chi (F_f)=\chi (F_g)=0$, by Proposition \ref{mult:prop:Randell_formula}, we have
$$(d-1)^{n+1}- \mu'(f)(d-1)+(-1)^{n}- \mu'(f)=0$$
and
$$(k-1)^{n+1}- \mu'(f)(k-1)+(-1)^{n}- \mu'(g)=0.$$
Thus, we define the polynomial map $P\colon\R\to \R$ by 
$$
P(t)=t^{n+1}- \mu'(f)t+(-1)^{n} - \mu'(f), \quad \forall t\in \R.
$$
Since $\mu'(f)=\mu'(g)$ (see Proposition \ref{mult:prop:top_inv_transv_milnor_number}), then $d-1$ and $k-1$ are positive zeros of $P$.
Since $\mu'(f)=\mu'(g)\geq 1$, by Descartes' Rule\index{Descartes' Rule}, $P$ has at most one positive zero. Thus, $d=k$.
\end{proof}
%

As a consequence, we obtain the result proved by Fernandes and Sampaio in \cite{FernandesS:2016} which says that the multiplicity\index{multiplicity!of surface in $\C^3$} of surface singularities in $\C^3$ is invariant under bi-Lipschitz homeomorphisms.

\begin{corollary}\label{mult_surface_blow}
Let $f,g:(\C^3,0)\to (\C,0)$ be two complex analytic functions. If $\varphi:(\C^3,V(f),0)\to (\C^3,V(g),0)$ is a bi-Lipschitz homeomorphism\index{homeomorphism!bi-Lipschitz}, then $m(V(f),0)=m(V(g),0)$.
\end{corollary}

\subsection{Invariance of the multiplicity under semi-bi-Lipschitz homeomorphisms on the functions}
In this section, we define some notions of equivalence on germs of functions and we prove some results on invariance of the multiplicity under those equivalences.
\begin{definition}\label{def_weakly_rugose}
We say that two germs of analytic functions $f,g:(\C^n,0)\to (\C^m,0)$ are {\bf semi-bi-Lipschitz equivalent}\index{equivalent!semi-bi-Lipschitz}, if there are constants $C_1,C_2>0$ and a germ of bijection $\varphi:(\C^n,0)\to (\C^n,0)$  such that 
\begin{enumerate}
\item [{\rm (1)}] $\frac{1}{C_1}\|x\|\leq \|\varphi(x)\|\leq C_1\|x\|$, for all small enough  $x\in \C^n$;
\item [{\rm (2)}] $\frac{1}{C_2}\|f(x)\|\leq \|g\circ\varphi(x)\|\leq C_2\|f(x)\|,$ for all small enough $x\in \C^n$.
\end{enumerate}
\end{definition}

\begin{definition}\label{def:inv_delta}
Let $U\subset \C^n$ be an open set such that $0\in U$ and let $f\colon U\to \C$ be an analytic function. Then, for each $r>0$ such that $B_r(0)\subset U$, we define
$$
\delta_{r}(f)=\textstyle{\sup \{\delta ;\,\frac{|f(z)|}{\|z\|^{\delta}} \mbox{ is bounded on } B_r(0)\setminus \{0\}\}.}
$$
\end{definition}

Note that $\delta_r(f)$ does not depend on $r>0$ (see Exercise \ref{mult:exer:no_depend_on_r}). Thus, we define this common number by $\delta(f)$.
\begin{proposition}\label{expoent}
Let $f,g\colon (\C^n,0)\to \C$ be a germ of an analytic function. Then, ${\rm ord}_0(f)=\delta(f)$.
\end{proposition}
\begin{proof}
If $\delta>m:={\rm ord}_0(f)$ and $f=f_m+f_{m+1}+...$ with $f_m\not=0$, then we choose $v\not\in V(f_m)$. Thus, $\lim \limits _{t\to 0^+ }\frac{|f(tv)|}{t^{\delta}}=+\infty $. Then, $\delta(f)\leq {\rm ord}_0(f)$.

If $\delta<m$, then $\lim \limits _{z\to 0}\frac{|f(z)|}{\|z\|^{\delta}}=0.$ Thus, there exists $r>0$ such that $\frac{|f(z)|}{\|z\|^{\delta}}\leq 1$, for all $z\in   B_r(0)$. This implies $\delta(f)\geq {\rm ord}_0(f)$.

Therefore, $\delta(f)= {\rm ord}_0(f)$.
\end{proof}

The next result is due to Comte, Milman and Trotman \cite{ComteMT:2002}.
\begin{theorem}\label{weakly_rugose}
Let $f,g\colon (\C^n,0)\to \C$ be two germs of analytic functions. If $f$ and $g$ are semi-bi-Lipschitz equivalent\index{equivalent!semi-bi-Lipschitz}, then ${\rm ord}_0(f)={\rm ord}_0(g)$.
\end{theorem}
\begin{proof}
By hypothesis, there are open neighbourhoods $U,\widetilde U$ of $0\in \C^n$, constants $C_1,C_2>0$ and a bijection $\varphi\colon U\to \widetilde U$  such that
\begin{enumerate}
\item [{\rm (1)}] $\frac{1}{C_1}\|x\|\leq\|\varphi(x)\|\leq C_1\|x\|$, for all $x\in U$;
\item [{\rm (2)}] $\frac{1}{C_2}\|f(x)\|\leq \|g\circ\varphi(x)\|\leq C_2\|f(x)\|, \quad \forall x\in U.$
\end{enumerate}
Let $\delta <\delta(g)$. Thus, we may take $\widetilde{r}>0$ such that $\frac{|g(z)|}{\|z\|^{\delta}}$ is bounded on $B_{\widetilde{r}}(0)\setminus \{0\}$, $B_{\widetilde{r}}(0)\subset \widetilde{U}$ and $B_{r}(0)\subset U$, where $r=\frac{\widetilde{r}}{C_1}$. In particular, $\varphi(B_{r}(0))\subset B_{\widetilde{r}}(0)$.

Moreover, we have 
\begin{eqnarray*}
\frac{|f(x)|}{\|x\|^{\delta}}& = & \frac{|f(x)|}{\|\varphi(x)\|^{\delta}}\frac{\|\varphi(x)\|^{\delta}}{\|x\|^{\delta}}\\
							 & \leq & C_1C_2\frac{|g(\varphi(x))|}{\|\varphi(x)\|^{\delta}},
\end{eqnarray*}
for all $x\in B_{r}(0)\setminus \{0\}$. Since $\frac{|g(z)|}{\|z\|^{\delta}}$ is bounded on $B_{\widetilde{r}}(0)\setminus \{0\}$, then $\frac{|f(x)|}{\|x\|^{\delta}}$ is bounded on $B_{r}(0)\setminus \{0\}$. This implies 
$$
\textstyle{\{\rho ;\,\frac{|g(z)|}{\|z\|^{\rho}} \mbox{ is bounded on } B_{\widetilde{r}}(0)\setminus \{0\}\}\subset \{s ;\,\frac{|f(x)|}{\|x\|^{s}} \mbox{ is bounded on } B_{r}(0)\setminus \{0\}\},}
$$
Then, we obtain $\delta_{\widetilde{r}}(g)\leq \delta_r(f)$ and, since $\delta_r(f)=\delta(f)$ and $\delta_{\widetilde{r}}(g)=\delta(g)$, we have $\delta(g)\leq \delta(f)$. Therefore, by Proposition \ref{expoent}, ${\rm ord}_0(g)\leq{\rm ord}_0(f)$. Similarly, we obtain ${\rm ord}_0(f)\leq{\rm ord}_0(g)$. Thus, we have the equality ${\rm ord}_0(g)={\rm ord}_0(f)$.
\end{proof}

\begin{definition}
We say that two germs of analytic functions $f,g\colon (\C^n,0)\to \C$ are:
\begin{itemize}
 \item {\bf bi-Lipschitz right equivalent}\index{equivalent!bi-Lipschitz right}, if there is a bi-Lipschitz homeomorphism $\varphi\colon(\C^n,0)\to (\C^n,0)$ such that $f(x)=g\circ\varphi(x),$ for all small enough  $x\in \C^n$;
 \item {\bf bi-Lipschitz right-left equivalent}\index{equivalent!bi-Lipschitz right-left}, if there are bi-Lipschitz homeomorphisms $\varphi\colon(\C^n,0)\to (\C^n,0)$ and $\phi\colon(\C,0)\to (\C,0)$ such that $f(x)=\phi\circ g\circ\varphi(x),$ for all small enough  $x\in \C^n;$
 \item {\bf rugose equivalent}\index{equivalent!rugose}, if there are constants $C_1,C_2>0$ and a germ of bijection $\varphi:(\C^n,0)\to (\C^n,0)$  such that 
\begin{enumerate}
\item [{\rm (1)}] $\frac{1}{C_1}\|x-y\|\leq \|\varphi(x)-\varphi(y)\|\leq C_1\|x-y\|$, for all small enough  $(x,y)\in \C^n\times f^{-1}(0)$;
\item [{\rm (2)}] $\frac{1}{C_2}\|f(x)\|\leq \|g\circ\varphi(x)\|\leq C_2\|f(x)\|,$ for all small enough $x\in \C^n$;
\end{enumerate}
\item {\bf bi-Lipschitz contact equivalent}\index{equivalent!bi-Lipschitz contact}, if there are a constant $C>0$ and a germ of bi-Lipschitz homeomorphism $\varphi\colon(\C^n,0)\to (\C^n,0)$  such that 
$$\frac{1}{C}\|f(x)\|\leq \|g\circ\varphi(x)\|\leq C\|f(x)\|,$$
 for all small enough $x\in \C^n$.
\end{itemize}

\end{definition}

The following result is a
direct consequence from the definitions.
\begin{proposition}\label{implica}
Let $f,g\colon (\C^n,0)\to \C$ be two germs of analytic functions. Let us consider the following statements:
\begin{enumerate}
\item [{\rm (1)}] $f$ and $g$ are bi-Lipschitz right equivalent;
\item [{\rm (2)}] $f$ and $g$ are bi-Lipschitz right-left equivalent;
\item [{\rm (3)}] $f$ and $g$ are bi-Lipschitz contact equivalent;
\item [{\rm (4)}] $f$ and $g$ are rugose equivalent;
\item [{\rm (5)}] $f$ and $g$ are semi-bi-Lipschitz equivalent.
\end{enumerate}
Then, $(1)\Rightarrow (2)\Rightarrow (3)\Rightarrow (4)\Rightarrow (5)$.
\end{proposition}
\begin{proof}
The proof is left as an exercise for the reader.
\end{proof}
We finish this subsection by stating some direct consequences of Theorem \ref{weakly_rugose} and Proposition \ref{implica}.
\begin{corollary}[See \cite{RislerT:1997}]
Let $f,g:\C^n\to \C$ be two germs of analytic functions. If $f$ and $g$ are rugose equivalent, then ${\rm ord}_0(f)={\rm ord}_0(g)$.
\end{corollary}
\begin{corollary}\label{mult:cor:c_lip_inv_mult}
Let $f,g\colon (\C^n,0)\to \C$ be two germs of analytic functions. If $f$ and $g$ are bi-Lipschitz contact equivalent at infinity, then ${\rm ord}_0(f)={\rm ord}_0(g)$.
\end{corollary}
\begin{corollary}\label{mult:cor:rl_lip_inv_mult}
Let $f,g\colon (\C^n,0)\to \C$ be two germs of analytic functions. If $f$ and $g$ are bi-Lipschitz right-left equivalent, then ${\rm ord}_0(f)={\rm ord}_0(g)$.
\end{corollary}
\begin{corollary}\label{mult:cor:r_lip_inv_mult}
Let $f,g\colon (\C^n,0)\to \C$ be two germs of analytic functions. If $f$ and $g$ are bi-Lipschitz right equivalent, then ${\rm ord}_0(f)={\rm ord}_0(g)$.
\end{corollary}


\subsection{Lipschitz invariance of multiplicity when the Lipschitz constants are close to 1.}

Let us remind the result proved by Draper in \cite{Draper:1969} which says that the multiplicity of complex analytic set $X$ at a point $p$ is the density\index{density} of $X$ at $p$.
\begin{theorem}[Draper \cite{Draper:1969}]\label{mult:thm:Draper_thm}
Let $Z\subset\C^k$ be a pure $d$-dimensional complex analytic subset with $0\in Z$. Then
$$
m(Z,0)=\lim\limits_{r\to0^+}\frac{\mathcal{H}^{2d}(Z\cap \overline{B}_r^{2k}(0))}{\mu_{2d}r^{2d}}
$$
where $\mathcal{H}^{2d}(Z\cap \overline{B}_r^{2k}(0))$ denotes the $2d$-dimensional Hausdorff measure\index{measure!Hausdorff} of $Z\cap \overline{B}_r^{2k}(0)=\{x\in Z;\|x\|\leq r\}$ and $\mu_{2d}$ is the volume\index{volume} of $2d$-dimensional unit ball\index{ball!unit}.
\end{theorem}
By using the above equality, we obtain the following:
\begin{proposition}\label{mult:prop:relation_mult_lip_constants}
Let $X\subset \C^n$ and $Y\subset \C^m$ be two germs at $0$ of complex analytic sets with $\dim X=\dim Y=d$. If there are constants $C_1,C_2>0$ and a bi-Lipschitz homeomorphism\index{homeomorphism!bi-Lipschitz} $\varphi\colon (X,0)\to (Y,0)$ such that 
$$
\frac{1}{C_1}\|x-y\|\leq \|\varphi(x)-\varphi(y)\|\leq C_2\|x-y\|, \quad \forall x,y \in X
$$
then 
$$
\frac{1}{(C_1C_2)^{2d}}m(X,0)\leq m(Y,0)\leq (C_1C_2)^{2d}m(X,0).
$$
\end{proposition}
\begin{proof}
The proof is left as an exercise for the reader.
\end{proof}

As a consequence of the above proposition, we obtain a result proved by Comte in \cite[Theorem 1]{Comte:1998}.
\begin{theorem}\label{mult:thm:const_close_one}
Let $X\subset \C^n$ and $Y\subset \C^m$ be two germs at $0$ of complex analytic sets with $\dim X=\dim Y=d$ and $M=\max \{m(X,0),m(Y,0)\}$. If there are constants $C_1,C_2>0$ and a bi-Lipschitz homeomorphism\index{homeomorphism!bi-Lipschitz} $\varphi\colon (X,0)\to (Y,0)$ such that 
$$
\frac{1}{C_1}\|x-y\|\leq \|\varphi(x)-\varphi(y)\|\leq C_2\|x-y\|, \quad \forall x,y \in X
$$
and $(C_1C_2)^{2d}\leq 1+\frac{1}{M}$, then $m(X,0)=m(Y,0).$
\end{theorem}

In fact, we obtain a slight better result than Theorem \ref{mult:thm:const_close_one}.

\begin{theorem}\label{mult:thm:better_const_close_one}
Let $X\subset \C^n$ and $Y\subset \C^m$ be two germs at $0$ of complex analytic sets with $\dim X=\dim Y=d$. Let $X_1,\dots,X_r$ and $Y_1,\dots,Y_s$ be the irreducible components\index{components!irreducible} of the tangent cones\index{tangent cones} $C(X,0)$ and $C(Y,0)$, respectively and let $M=\max \{m(X_1,0),...,m(X_r,0),m(Y_1,0),...,m(Y_s,0)\}$. If there are constants\index{constants} $C_1,C_2>0$ and a bi-Lipschitz homeomorphism\index{homeomorphism!bi-Lipschitz} $\varphi\colon (X,0)\to (Y,0)$ such that 
$$
\frac{1}{C_1}\|x-y\|\leq \|\varphi(x)-\varphi(y)\|\leq C_2\|x-y\|, \quad \forall x,y \in X
$$
and $(C_1C_2)^{2d}\leq 1+\frac{1}{M}$, then $m(X,0)=m(Y,0).$
\end{theorem}
\begin{proof}
By the bi-Lipschitz invariance of the tangent cones\index{tangent cones!bi-Lipschitz invariance} (see Theorem \ref{tg:thm:equivcone}), there is a global bi-Lipschitz homeomorphism\index{homeomorphism!global bi-Lipschitz} $d\varphi\colon C(X,0)\to C(Y,0)$ such that $d\varphi(0)=0$ and
$\frac{1}{C_1}\|v-w\|\leq \|d\varphi(v)-d\varphi(w)\|\leq C_2\|v-w\|, \quad \forall v,w \in C(X,0).$

By Theorem \ref{mult:thm:inv_multiplicities}, $r=s$ and, up to a re-ordering of indices, $k_X(X_j)=k_Y(Y_j)$ and $Y_j=d\varphi(X_j)$, $\forall \ j$. Moreover, by Proposition \ref{mult:prop:kurdyka-raby}, we obtain
$$m(X,0)=\sum\limits_{j=1}^r k_X(X_j)\cdot m(X_j,0)$$
and 
$$m(Y,0)=\sum\limits_{j=1}^r k_Y(Y_j)\cdot m(Y_j,0).$$
Since $X_j$ and $Y_j$ are homogeneous algebraic sets\index{algebraic sets!homogeneous}, we have ${\rm deg}(X_j)=m(X_j,0)$ and ${\rm deg}(Y_j)=m(Y_j,0)$. By Theorem \ref{mult:thm:const_close_one}, $m(X_j,0)=m(Y_j,0)$ for all $j$. Therefore,
$m(X,0)=m(Y,0).$
\end{proof}

\subsection{Question AL(2) and final comments}

Let us recall the Question AL($d$).

\begin{enumerate}[leftmargin=0pt]
	\item[]{\bf Question AL($d$)} Let $X\subset \C^n$ and $Y\subset \C^m$ be two complex analytic sets with $\dim X=\dim Y=d$, $0\in X$ and $0\in Y$. 
	If there exists a bi-Lipschitz 
	homeomorphism\index{homeomorphism!bi-Lipschitz} $\varphi\colon (X,0)\to (Y,0)$, then is $m(X,0)=m(Y,0)$?
\end{enumerate}

We finish this section bringing a complete answer to this question. Next result is a positive answer to that. Its proof was published in \cite{BobadillaFS}. It is valuable to mention that Neumann and Pichon  also got a positive answer to question AL(2) with the additional hypothesis that the considered surface singularities are normal (see \cite{NeumannP:2012}).
 
\begin{theorem}\label{thm:multiplicity_surfaces}
Let $X\subset \C^{N+1}$ and $Y\subset \C^{M+1}$ be two complex analytic surfaces. 
If $(X,0)$ and $(Y,0)$ are bi-Lipschitz homeomorphic\index{homeomorphic!bi-Lipschitz}, then $m(X,0)=m(Y,0)$.
\end{theorem}

In view of what we have already proved, mainly Theorem \ref{mult:thm:reduction_to_homogeneous} which says that it is enough to addresses Question AL($d$) for irreducible homogeneous singularities, we have that the above theorem is an immediate consequence of the following result proved in \cite{BobadillaFS}.

\begin{proposition}\label{mult:prop:homogeneous_surfaces}
	Let $S\subset\C^n$ be a 2-dimensional homogeneous and irreducible algebraic set{algebraic sets!irreducible}. Then the multiplicity $m(S,0)=m$ is given by the following: the torsion part of $H^2(S\setminus 0,\Z)$ is isomorphic to $\Z /m\Z$. In particular, if $S,S'$ are 2-dimensional homogeneous and irreducible algebraic sets such that $(S,0)$ and $(S',0)$ are homeomorphic, then $m(S,0)=m(S',0)$. 
\end{proposition}

\begin{proof}[Proof of Theorem \ref{thm:multiplicity_surfaces}] 
	In order to prove this theorem, as we saw  above, by Theorem \ref{mult:thm:reduction_to_homogeneous}, it is enough to assume that $X\subset \C^{N+1}$ and $Y\subset \C^{M+1}$ are 2-dimensional homogeneous and irreducible algebraic sets{algebraic sets!homogeneous}. Since a bi-Lipschitz homeomorphism\index{homeomorphism!bi-Lipschitz} between $(X,0)$ and $(Y,0)$ is also a homeomorphism between them, by Proposition \ref{mult:prop:homogeneous_surfaces}, it follows that $m(X,0)=m(Y,0)$.  
\end{proof}

Finally, we complete the answer of Question AL($d$) by showing that, for dimension $d$ greater then 2, we always have singularities which are bi-Lipschitz homeomorphic with different multiplicities. This result was proved in \cite{BirbrairFSV:2020}.
 
\begin{theorem}\label{_Main_example_CP^1_times_CP^1_Theorem_}
For each $n\geq 3$, there exists a family $\{ Y_i \}_{i\in\N}$  of $n$-dimensional complex algebraic varieties
$Y_i \subset \C^{n_i+1}$ such that:
\begin{itemize}
\item [(a)] for each pair $i\neq j$, the germs at the origin of $Y_i\subset \C^{n_i+1}$ and $Y_j\subset \C^{n_j+1}$ are bi-Lipschitz equivalent\index{equivalent!bi-Lipschitz}, but $(Y_i,0)$ and $(Y_j,0)$ have different multiplicity.
\item [(b)] for each pair $i\neq j$, there are $n$-dimensional complex algebraic varieties $Z_{ij}, \widetilde Z_{ij}\subset \C^{n_i+n_j+2}$ such that $(Z_{ij},0)$ and $(\widetilde Z_{ij},0)$ are ambient bi-Lipschitz equivalent\index{equivalent!ambient bi-Lipschitz}, but  $m(Z_{ij},0)=m(Y_{i},0)$ and $m(\widetilde Z_{ij},0)=m(Y_{j},0)$ and, in particular, they have different multiplicity\index{multiplicity!different}.
\end{itemize}
\end{theorem}

\begin{proof}[Sketch of the proof.]
	Let $\{p_i\}_{i\in\N}$ be the family of odd prime numbers\index{number!odd prime}. For each $i\in\N$,
	let $X_i$ be an embedding of $\C P^1\times\C P^1$ into $\C P^{n_i}$ with degree $4p_i= 2\cdot2\cdot p_i$ (we can do this by using a composition of Segre and Veronese embeddings\index{embeddings!Segre and Veronese}). Let $Y_i\subset\C^{N+1}$ be the respective affine algebraic complex cone associated to $X_i$ and $S_i$ its the respective link at the origin of $\C^{n_i+1}$. More precisely, $S_i=Y_i\cap\mathbb{S}^{2n_i+1}$.
	By construction, $m(Y_i,0)=4p_i$ for all $i\in\N$. On the other hand, it was proved in \cite{BirbrairFSV:2020} that $S_i$ is diffeomorphic\index{diffeomorphic} to $\mathbb{S}^2\times\mathbb{S}^3$ for all $i\in\N$ and, then $(Y_i,0)$ is bi-Lipschitz homeomorphic\index{homeomorphic!bi-Lipschitz} to $(Y_i,0)$ while $m(Y_i,0)\neq m(Y_i,0)$ for all $i\neq j$. Hence, Item (a) is proved.
	
	Concerning to Item (b), let $f_{ij} \colon Y_i \rightarrow  Y_j$ be a bi-Lipschitz homeomorphism such that $f_{ij} (0) = 0$ (we are considering $f_{ij}^{-1}=f_{ji}$). Let $F_{ij} \colon \C^{n_i+1} \rightarrow \C^{n_j+1}$  be a Lipschitz extension\index{Lipschitz extension} of $f_{ij}$  (see Lemma \ref{mcshane}). Let us define $\phi,\psi \colon  \C^{n_i+1} \times \C^{n_j+1}\rightarrow \C^{n_i+1} \times \C^{n_j+1}$ as follows:
	$$\phi(x, y) = (x - F_{ji} (y + F_{ij} (x)), y + F_{ij} (x))$$
	and
	$$\psi(z, w) = (z + F_{ji} (w), w - F_{ij} (z + F_{ji} (w))).$$
	It easy to verify that $\phi$ and $\psi$ are inverse maps\index{maps!inverse} of each other. Since $\phi$ and $\psi$ are composition of Lipschitz maps\index{maps!composition of Lipschitz}, they are also Lipschitz maps. Moreover, if we denote $Z_{ij} = Y_i \times \{0\}$ and
	$\widetilde{Z}_{ij} = \{0\} \times Y_j$, we get $\phi(Z_{ij} ) = \widetilde{Z}_{ij}$ (see \cite{Sampaio:2016}). Therefore, $(Z_{ij} , 0)$
	and $(\widetilde{Z}_{ij} , 0)$ are ambient bi-Lipschitz equivalent\index{equivalent!ambient bi-Lipschitz}, while $m(Z_{ij} , 0) = m(Y_i, 0)$ and
	$m(\widetilde{Z}_{ij} , 0) = m(Y_j, 0)$ are different.
\end{proof}

\subsection{Exercises}\label{sec:exercises_multiplicity}

\begin{exercise}\label{mult:exer:no_depend_on_r}
Prove that $\delta_r(f)$ as in Definition \ref{def:inv_delta} does not depend on $r>0$.
\end{exercise}

\begin{exercise}
Prove Proposition \ref{implica}.
\end{exercise}

\begin{exercise}
Give a direct proof to Corollary \ref{mult:cor:rl_lip_inv_mult}.
\end{exercise}

\begin{exercise}
 Prove Proposition \ref{mult:prop:relation_mult_lip_constants}.
\end{exercise}






\printindex

\end{document}